\documentclass[11pt,a4paper]{amsart}
\usepackage[latin1]{inputenc}
\usepackage[francais,english]{babel}
\usepackage{amsmath}
\usepackage{amsthm}
\usepackage{amsfonts}
\usepackage{amssymb}
\usepackage{amscd}
\usepackage{mathrsfs}
\usepackage{graphicx}
\usepackage[left=3cm,right=3cm,top=3cm,bottom=3cm]{geometry}
\author{Jean Barbet-Berthet}
\thanks{We would like to thank the IRMA of Strasbourg (France) for granting us access to their scientific library}
\title{Algebraic equivarieties over a commutative field}

\theoremstyle{plain}							
\newtheorem{thm}{\textbf{Theorem}}[section]
\newtheorem{prop}[thm]{Proposition}
\newtheorem{cor}[thm]{Corollary}

\newtheorem{lem}[thm]{Lemma}

\theoremstyle{definition}						

\newtheorem{defi}[thm]{Definition}
\newtheorem{rem}[thm]{Remark}
\newtheorem{ex}[thm]{Example}

\theoremstyle{remark}							

\renewcommand{\and}{\wedge}						

\newcommand{\onto}{\twoheadrightarrow}			
\newcommand{\into}{\hookrightarrow}				

\newcommand{\xx}{\times}

\renewcommand{\phi}{\varphi}					

\newcommand{\ms}{\mathscr}						
\newcommand{\mf}{\mathfrak}						
\newcommand{\mbb}{\mathbb}						
\newcommand{\ov}{\overline}						
\newcommand{\m}[1]{\mathsf{#1}} 				
\newcommand{\wt}{\widetilde}					

\newcommand{\subs}{\subseteq}

\renewcommand{\m}{\mf m}							

\renewcommand{\>}{\rangle}
\renewcommand{\o}{\otimes}								

\renewcommand{\O}{\ms O}								

\begin{document}
\renewcommand{\proofname}{Proof}

\maketitle

\begin{abstract}
We expand our previously founded basic theory of equiresidual algebraic geometry over an arbitrary commutative field, to a well-behaved theory of (equiresidual) algebraic varieties over a commutative field, thanks to the generalisation of affine algebraic geometry by the use of canonical localisations and $*$-algebras. We work here in an equivalent and more suggestive \og concrete" setting with structural sheaves of functions into the base field, which allows us to give a set-theoretic description of the products of equivarieties in general. Locally closed subvarieties are naturally equipped with an equivariety structure as in the particular case of algebraically closed fields, and this allows in particular to embed all quasi-projective (equi)varieties in general into the category of algebraic (equi)varieties.
\end{abstract}

\section{Introduction}
In \cite{EQAG}, we have laid the very first foundations for algebraic geometry over any commutative field, thereby generalising the situation over an algebraically closed field in an affine setting. In particular, we have proved an equiresidual generalisation of Hilbert's Nullstellensatz (the \og Äquinullstellensatz, Theorem 2.4), we have characterised the local sections of the sheaf of regular functions over any subvariety of an affine space (\cite{EQAG}, Theorem 3.17), and we have established a duality between affine algebraic equivarieties and a perfectly identifed category of \og affine $*$-algebras" (\cite{EQAG}, Theorem 4.15). In the present work we finish laying the first foundations of equiresidual algebraic geometry by extending the affine theory to a theory of \emph{algebraic equivarieties}, which generalise algebraic varieties (in the sense of locally ringed spaces locally isomorphic to \emph{maximal} spectra of affine algebras) in the \og absolute case" (i.e. over algebraically closed fields), and encompass all quasi-projective (equi)varieties over an arbitrary commutative field, i.e. isomorphic to locally closed subvarieties of projectives spaces. This establishes the sturdiness of the basic theory, especially manifest in the non-trivial use of canonical localisations in the structure theorem for projective spaces (Theorem \ref{PROJSP}). The connecting thread of the development is the use of products of equivarieties in order to characterise the \emph{separatedness} of an equivariety by the closedness of the diagonal in its square (Proposition \ref{MAG.4.24}), closely following the perspective of \cite{MAG}, Chapter 4.\\
In section \ref{OPAFF}, we define the product of two affine algebraic equivarieties, by using the duality between affine (algebraic) equivarieties and affine $*$-algebras of \cite{EQAG} and introducing the sum of affine $*$-algebras. We provide a concrete description of equivarieties in  general which allows us to give an explicit definition of the \emph{Zariski} topology on the set-theoretic product of two affine equivarieties, as well as an explicit description of a sheaf over this product, which make it a concrete product of equivarieties.
In section \ref{PRODEQ}, we use the product of affine algebraic equivarieties in order to define a concrete product of equivarieties in general, thanks to a glueing property of affine equivarieties. 
In section \ref{SUBVAR}, we introduce embeddings, immersions and subvarieties of concrete equivarieties in general, with a concrete description of the restriction of the structural sheaf to subvarieties. We consider the notion of a locally closed subvariety, which is naturally equipped with the structure of an equivariety. These appear ubiquitary as open or closed subvarieties and graphs of regular maps, and will be our basic concept for understanding quasi-projective equivarieties.
In section \ref{SEPALG}, we introduce the separation property for equivarieties, which is as in the absolute case a well-known analogue of the Hausdorff property for topological spaces, and the algebraic equivarieties. It is characterised, thanks to the product of equivarieties, by the closedness character of the diagonal subvariety of the square. We also give a local characterisation of separatedness which will be useful for the study of projective spaces.
In section \ref{LOCVAR}, we show that projectives spaces are algebraic equivarieties, thereby generalising the usual description over algebraically closed fields. We also show in general that locally closed subvarieties of separated (algebraic) equivarieties are separated (algebraic), which allows us to establich that all quasi-projective equivarieties are algebraic (Theorem \ref{LOCSEP}).\\

We adopt the conventions and notations of \cite{EQAG}. As already mentionned, the present theory closely follows Milne's synthesis on the subject over algebraically closed fields in \cite{MAG}, and heavily relies on generalisations of his Chapter 4.

\section{Products of affine algebraic equivarieties}\label{OPAFF}
\subsection*{Sums of affine $*$-algebras}
We want to define the product of two affine (algebraic) equivarieties. For this purpose, we may use the duality between affine equivarieties over $k$ and affine $*$-algebras (\cite{EQAG}, Theorem 4.15), if we can identify the sum of two affine $*$-algebras. In fact, we have the

\begin{prop}\label{MAG.4.15}
If $A$ and $B$ are two affine $*$-algebras over $k$, then $A\o_k B$ is special and so is $A*B=(A\o_k B)_M$.
\end{prop}
\begin{proof}
We adapt and generalise the proof of \cite{MAG}, Proposition 4.15. We want to show, by \cite{EQAG}, Lemma 3.7, that for all $D(\ov x)\in \ms D$ and $\ov\beta,\alpha\in A\o_k B$ such that $D^\#(\ov \beta,\alpha)=0$, we have $\alpha=0$. Write $\alpha=\sum_{i=1}^n a_i\o b_i\in A\o_k B$ : we may always rewrite this expression in such a way that the $b_i$'s are linearly independent over $k$ (see \cite{MAG}). If $\m$ is a maximal ideal of $A$, we consider the canonical map $A\to A/\m$, $a\mapsto [a]$ : as $A$ is an affine $*$-algebra, the structural morphism $k\to A/\m$ is an isomorphism, and identifying those two fields we get a map $A\o_k B\to k\o_k B\cong B$, $a\o b\mapsto [a]\o b\mapsto [a] b$. By hypothesis, we have $D^\#([\ov \beta],[\alpha])=0$ in $B$, and as $B$ is special, we must have $[\alpha]1=\sum_i [a_i]b_i=0$ by \cite{EQAG}, Lemma 3.7; as the $b_i$'s are linearly independent over $k$, we get $[a_i]=0$ for all $i$ : all the $a_i$'s are in every maximal ideal of $A$, so as $A$ is itself special, by Lemma 4.13 of \cite{EQAG} we get $a_i=0$ for all $i$, whence $\alpha=0$ : by \cite{EQAG}, Lemma 3.7, $A\o_k B$ is special. Now $A*B=(A\o_k B)_M$ is special by \cite{EQAG}, Proposition 3.9.
\end{proof}
\begin{rem}
Using similar ideas and the characterisation of irreducible affine algebraic equivarieties by their algebras of global sections being integral domains, it is easy to check that if $A$ and $B$ are integral affine $*$-algebras, then $A*B$ is integral as well.
\end{rem}

\noindent If $A$ and $B$ are two $*$-algebras over $k$, we will from now on write $A*B:=(A\o_k B)_M$. By the universal properties of the tensor product and the canonical localisation, it is obviously a sum in the category of $*$-algebras; by Proposition \ref{MAG.4.15} if $A$ and $B$ are affine $*$-algebras, then $A*B$ is affine as well, so that $A*B$ is the sum of $A$ and $B$ in the category $*Aff_k$. Now consider two affine equivarieties $(V,\O_V)$ and $(W,\O_W)$ over $k$ : by definition, $J(V)=\Gamma(V,\O_V)$ and $J(W)=\Gamma(W,\O_W)$ are affine $*$-algebras over $k$ and by what precedes, 
their sum $J(V)*J(W)=(J(V)\o_k J(W))_M$ gives us a direct construction of the product of $V$ and $W$ in $EVar^a_k$ : by duality, thanks to \cite{EQAG}, Proposition 4.9, $Spm\ J(V) * J(W)$, with its natural structural sheaf, is such a product.

\subsection*{Concrete equivarieties}
However, we want to give a more workable description of a product of $V$ and $W$, with underlying set the Cartesian product $V\xx W$. It will technically be most convenient to restrict ourselves to the following general kind of equivarieties.
\begin{defi}
i) Say that an equivariety $(V,\O_V)$ over $k$ is \emph{concrete}, if for every open subset $U\subs V$, the $k$-algebra $\O_V(U)$ is a subset of the $k$-algebra $k^U$ of all functions $U\to k$.\\
ii) If $V$ and $W$ are concrete equivarieties, say that a map $f:V\to W$ is \emph{regular}, if it is continuous and for every open subset $U\subs W$ and every $s\in \O_W(U)$, we have $s\circ f\in \O_V(f^{-1} U)$.
\end{defi}
\begin{rem}
i) By definition of an equivariety, for every open subset $U$ of a concrete equivariety $V$ the members of $\O_V(U)$ are \emph{continuous} functions.\\
ii) By extension, if $f:V\to W$ is a regular map of concrete equivarieties, we will implicitly consider $f$ as the morphism of equivarieties $(f,f^\#)$, where $f^\#:\O_V\to f_*\O_V$ is defined by $f^\#_U:s\in \O_W(U)\mapsto s\circ f$ for every open $U\subs W$.
\end{rem}

\begin{lem}\label{OPENBASIS}
If $V$ is an equivariety, then the affine open subsets of $V$ are a basis of open subsets of $V$.
\end{lem}
\begin{proof}
Let $U\subs V$ be an open subset, and $P\in U$ : by definition of an equivariety, there exists an open neighbourhood $O$ of $P$ in $V$ such that $(O,\O_V|_O)$ is an affine algebraic equivariety. Now we have $P\in O\cap U$, an open subset of $O$ : as such $O\cap U$ is a finite union of open affine subsets of $O$ (this is true for subvarieties of affine spaces, so also for every affine algebraic equivariety), and thus of $V$, so there exists an affine open subset $W\subs V$ such that $P\in W\subs O\cap U\subs U$, and therefore the affine open subsets of $V$ are a basis of open subsets.
\end{proof}

\noindent As it is implicit in the definition of an equivariety, the restriction to concrete ones is not an essential one, thanks to the following

\begin{prop}\label{CONCEQ}
Every morphism of concrete equivarieties is a regular map (and reciprocally), and every equivariety over $k$ is naturally isomorphic to a concrete equivariety.
\end{prop}
\begin{proof}
As for the first claim, let $(f,f^\#):V\to W$ be a morphism of equivarieties, i.e. of locally ringed spaces in $k$-algebras, with $V$ and $W$ concrete equivarities over $k$. By hypothesis, $f$ is continuous, and we want to show that for every open $U\subs W$, $f^\#_U(s)=s\circ f$; as both are functions $U\to k$, it suffices to show that for every $P\in U$, we have $f^\#_U(s)(P)=s\circ f(P)$. For $P\in U$, write $Q=f(P)$ : we know that the induced $k$-morphism $f^\#_P:\O_{W,Q}\to \O_{V,P}$ is local, and by definition we have $f^\#_P([s,U]_Q)=[f^\#_U(s),f^{-1}U]_P$. Let $i_P:k\to\O_{V,P}$ and $i_Q:k\to \O_{W,Q}$ be the structural morphisms on the stalks, and $j_P:k\cong\ov{\O_{V,P}}$ and $j_Q:k\cong \ov{\O_{W,Q}}$ the associated residual structural isomorphisms, so that $f^\#_P\circ i_Q=i_P$ and $\ov{f^\#_P}\circ j_Q=j_P$, for $\ov{f^\#_P}:\ov{\O_{W,Q}}\to\ov{\O_{V,P}}$ the residual $k$-isomorphism. We have $j_P(f^\#_U(s)(P))=\ov{[f^\#_U(s),f^{-1}U]_P}=\ov{f^\#_P}(\ov{[s,U]_Q})$ (by definition of $\ov{f^\#_P}$) $=\ov{f^\#_P}\circ j_Q(s(f(P)))=j_P(s\circ f(P))$, whence $f^\#_U(s)(P)=s\circ f(P)$, so $f^\#_U(s)=s\circ f$, and therefore $(f,f^\#)$ is regular.
As for the second claim, let $(V,\O_V)$ be an equivariety over $k$. By definition, $V$ has a basis of \og affine" open subsets, so let $U\subs V$ be an affine open subset : there exists an isomorphism $\phi:U\cong V_0\subs k^n$ with an affine subvariety, and we let $\O_V'(U):=\{s\circ \phi : s\in\O_{V_0}(V_0)\}$, a $k$-algebra of functions $U\to k$. We first check that this is well defined, so let $\psi:U\cong W_0\subs k^m$ be another isomorphism with an affine subvariety, and let $s\circ \phi\in \O_V'(U)$ : we have $t:=s\circ \phi\circ \psi^{-1}:W_0\to U\to V_0\to k$, and as $\phi\circ \psi^{-1}$ is an isomorphism and $s\in \O_{V_0}(V_0)$, by the first part of the proof we have $t\in \O_{W_0}(W_0)$, because $\phi\circ \psi^{-1}$ is regular. It follows that $t\circ \psi=s\circ \phi$, so $s\circ\phi\in \{t\circ \psi : t\in \O_{W_0}(W_0)\}$ and by symmetry, we get $\{s\circ \phi : s\in \O_{V_0}(V_0)\}=\{t\circ \psi : t\in \O_{W_0}(W_0)\}$, so that $\O_V'(U)$ is well defined. Now if $U\subs V$ is any open subset, write $U=\bigcup_i U_i$ as the union of \emph{all} affine open subsets $U_i$ of $U$, and let $\O_V'(U):=\{f:U\to k :\forall i,f|_{U_i}\in \O_V'(U_i)\}$ : this is well defined, and $\O'_V$ is obviously a sheaf of continuous functions. Indeed, if $U'\subs U$ is an open subset, then for the cover $U'=\bigcup_j U'_j$ of $U'$ by all affine open subsets, each $U'_j$ is one of the $U_i$'s, so the restriction $f|_{U'}$ is in $\O'_V(U')$; and as the presheaf $\O'_V$ is defined locally by concrete maps, it is clearly a sheaf. It remains to prove that $\O_V'\cong \O_V$ and it suffices to describe an isomorphism on the affine open subsets of $V$, so let $U\subs V$ be any such with $\phi:U\cong V_0\subs k^n$ a local chart. To $s\in \O_V(U)$ we associate $(\phi^\#_{V_0})^{-1}(s)\circ \phi\in \O_V'(U)$, and to every $t\circ \phi\in \O_V'(U)$ we associate $\phi_{V_0}^\#(t)\in \O_V(U)$; as $\phi$ is an isomorphism, this is well defined, and both maps are obviously mutually inverse $k$-isomorphisms. 
\end{proof}

\subsection*{The Zariski topology on the product}
Let $V$ be a \emph{concrete} affine equivariety over $k$ : there exists an isomorphism $\phi:V\cong W$ with an affine algebraic subvariety $W$ of $k^n$ say, whence an isomorphism $J(W)=\Gamma(W,\O_W)\cong \Gamma (V,\O_V)=J(V)$. It follows that the closed subsets of $V$ are the zero sets of subsets of $J(V)$, because this is the case for $W$. More precisely, if $F\subs W$ is a closed subset, there exists a subset $S$ of $k[W]$, such that $F=\ms Z_W(S)$ and as $k[W]$ embeds into $J(W)$, we may assume that $S\subs J(W)$ and $Z=\{P\in W :\forall f\in S,f(P)=0\}$; conversely, if $S\subs J(W)$, the set $\ms Z_W(S)=\{P\in W : \forall f\in S,f(P)=0\}$ is closed in $W$, by continuity of the elements of $J(W)$ (or alternatively because we have a natural $k$-isomorphism $J(W)\cong k[W]_M$). Transposing through $\phi$, we see that the closed subsets of $V$ have the same description.\\
\noindent Now let $W$ be another concrete affine equivariety. We define on the set-theoretic product $V\xx W$ the \emph{Zariski topology} as follows. The closed subsets of $V\xx W$ are taken as the sets of zeros of maps $f:V\xx W\to k$ of the form $(P,Q)\mapsto \sum_i g_i(P)h_i(Q)$ with $g_i\in J(V),h_i\in J(W)$. Such sets are obviously closed under arbitrary intersections, whereas if $X,Y\subs V\xx W$ are defined by $X=\ms Z_{V\xx W}(S)=\{(P,Q)\in V\xx W :\forall f\in S,f(P,Q)=0\}$ and $Y=\ms Z_{V\xx W}(T)$, with $S$ and $T$ sets of functions of the preceding form, we clearly have $X\cup Y=\ms Z_{V\xx W}(ST)$, where $ST$ is the set of all products $fg$ of functions, for $f\in S$, $g\in T$ : we have indeed defined a topology.
In fact, by the already cited duality theorem 4.15 of \cite{EQAG}, the $k$-algebra of sections of the structural sheaf on any product of $V$ and $W$ is canonically isomorphic to $J(V)*J(W)$. Now every element $f$ of $J(V)*J(W)$ has the form 
$$\frac{\sum_j g_j\o h_j}{\sum_l u_l\o v_l}$$ with $\sum_l u_l\o v_l\in M_{J(V)\o J(W)}$, and it may therefore be conceived as a map $$(P,Q)\in V\xx W\mapsto \frac{\sum_j g_j(P)h_j(Q)}{\sum_l u_l(P)v_l(Q)},$$ which is well-defined by definition of $M_{J(V)\o J(W)}$ (and is in fact the $k$-morphism obtained as the sum of the evaluation morphisms $J(V)\to k$ and $J(W)\to k$ at $P$ and $Q$). As the denominator in this last expression is everywhere nonzero, any zero set of a set of such \og abstract functions" has the form $F=\ms Z_{V\xx W}(S)=\{(P,Q)\in V\xx W : \forall f\in S,f(P,Q)=0\}$, for $S\subs J(V)*J(W)$. It is clear that we have defined in two different ways the functions $V\xx W\to k$ which serve as a definition for the Zariski topology on $V\xx W$; the abstract description will sometimes be useful.

\begin{rem}
In general, i.e. if $V$ and $W$ are not supposed to be concrete, the elements of $J(V)*J(W)$ still define functions $V\xx W\to k$; this strengthens the relevance of working with concrete equivarieties. 
\end{rem}

\subsection*{The concrete structural sheaf on the product}
Next we define a concrete structural sheaf on $V\xx W$. If $U\subs V\xx W$ is open, we say that a function $f:U\to k$ is \emph{regular at $(P,Q)\in U$} if there exists an open $U'\subs U$ and $g,h\in J(V)*J(W)$ such that $(P,Q)\in U'$ and for all $(P',Q')\in U'$, we have $h(P',Q')\neq 0$ and $f(P',Q')=g(P',Q')/h(P',Q')$. This obviouly defines a sheaf, which we note $\O_{V\xx W}$.

\begin{lem}\label{OPENPROJ}
Each projection map $V\xx W\to V,W$ is an open regular map.
\end{lem}
\begin{proof}
Write $\pi,\rho : V\xx W\to V,W$ the respective projection maps. If $U\subs V$ is open, there exists a subset $S$ of $J(V)$ such that $U=\{P\in V : \exists f\in S,f(P)\neq 0\}$, so $\pi^{-1}U=\{(P,Q)\in V\xx W : \exists f\in S,(f.1)(P,Q)\neq 0\}$ is open by definition of the Zariski topology on $V\xx W$, and $\pi$ is continuous. Furthermore, if $s\in \O_V(U)$ then each $P\in U$ has an open neighbourhood $U_P\subs U$ such that for every $P'\in U_P$, we have $s(P')=f(P')/g(P')$ for some $f,g\in J(V)$; as $U'_P=\pi^{-1}U_P$ is open and for every $(P',Q')\in U'_P$, we have $s\circ\pi(P',Q')=f(P')/g(P')$, we conclude that $s\circ \pi\in \O_{V\xx W}(\pi^{-1}U)$, and $\pi$ is regular; likewise, $\rho$ is regular. As for openness, it also suffices to prove it for $\pi$. Let $U\subs V\xx W$ be open : there exist a subset $S$ of $J(V)*J(W)$ such that $U=\{(P,Q)\in V\xx W : \exists f\in S,f(P,Q)\neq 0\}$; furthermore, we may assume that each member of $S$ has the form $\sum_i g_i *h_i$ (the denominators may be neglected for the definition of the topology). Let $X$ be the set of all elements of $J(V)$ of the form $\sum_i h_i(Q).g_i$, where $Q\in W$ and $\sum_i g_i*h_i\in S$ : if $P\in \pi(U)$, by definition there exists $Q\in W$ and $\sum_i g_i*h_i\in S$ such that $\sum_i g_i(P)h_i(Q)\neq 0$, so $P\in \{P'\in V : \exists f\in X,f(P')\neq 0\}$, an open subset of $V$. Conversely, let $P$ be a member of this set : by definition of $X$, there exists $Q\in W$ and $\sum_i g_i*h_i\in S$ such that $\sum_i g_i(P)*h_i(Q)=(\sum_i h_i(Q).g_i)(P)\neq 0$, so that $P\in \pi(U)$. It follows that $\pi(U)=\{P\in V : \exists f\in X,f(P)\neq 0\}$, and thus $\pi(U)$ is open, and $\pi$ is open. 
\end{proof}

\noindent We finally proceed to show that $(V\xx W,\O_{V\xx W})$ is a product in $EVar^a_k$ : it suffices to exhibit an isomorphism between $(V\xx W,\O_{V\xx W})$ and $X=Spm(J(V)*J(W))$. The underlying homeomorphism is $\mu : V\xx W \to X$, $(P,Q)\mapsto Ker(e_{P,Q})$, where $e_{P,Q}:J(V)*J(W) \to k$ is the \og evaluation at $(P,Q)$" obtained by the universal property of $J(V)*J(W)$, and which maps $f\in J(V)*J(W)$ to $f(P,Q)$ as defined above : as $J(V)*J(W)$ is an affine $*$-algebra, $X$ is naturally isomorphic as a set to $hom_k(J(V)*J(W),k)$, so $\mu$ is clearly a bijection, and by definition of the Zariski topology on $X$, it is clearly a homeomorphism. It remains to describe a sheaf isomorphism $\mu^\#:\O_X\cong \mu_* \O_{V\xx W}$ : if $U\subs X$ is open, $s\in \O_X(U)$ and $(P,Q)\in \mu^{-1}U$, there exists an open neighbourhood $U_{P,Q}\subs U$ of $\mu(P,Q)=\m_{P,Q}$ in $X$ and $f_{P,Q},g_{P,Q}\in J(V)*J(W)$ such that $s|_{U_{P,Q}}\equiv [f_{P,Q}]/[g_{P,Q}]$, and we let $t(P,Q):=f_{P,Q}(P,Q)/g_{P,Q}(P,Q)$ : it should be obvious by the definition of $\O_{V\xx W}$ that $t\in \O_{V\xx W}(\mu^{-1} U)$, and that $\mu^\#_U:s\mapsto t$ defines an isomorphism $\mu^\#:\O_X\cong \mu_*\O_{V\xx W}$, so we have the

\begin{prop}\label{AFFPROD}
The concrete equivariety $(V\xx W,\O_{V\xx W})$ is a product of $V$ and $W$ in the category $EVar^a_k$.
\end{prop}

\noindent As an essential application, if $V$ is a concrete affine algebraic equivariety, consider the diagonal $\Delta_V=\{(P,P)\in V^2 : P\in V\}$. We have $\Delta_V=\{(P,Q)\in V^2 : \forall f\in J(V), (f*1-1*f)(P,Q)=0\}$ : indeed, for each $P\in V$ we have $(f*1-1*f)(P,P)=f(P)-f(P)=0$, and if $(f*1-1*f)(P,Q)=0$ for all $f\in J(V)$, for every $f\in J(V)$ we have $f(P)=f(Q)$, whence $P=Q$, because this is obviously true for $V$ an affine subvariety of an affine space (apply to a set of generators of $k[V]$ in this case). By what precedes, the diagonal $\Delta_V$ is therefore closed in $V^2$ for the Zariski topology; this \og separatedness property" is an analogue of the Hausdorff property, which we will use in order to generalise affine algebraic equivarieties to what we call \og algebraic equivarieties", in order to encompass all quasi-projective equivarieties (locally closed subvarieties of projective spaces). For this purpose we have to define the product of two equivarieties in general, and we will need the following

\begin{lem}\label{UNIPROD}
If $V$ and $W$ are concrete affine algebraic equivarieties over $k$, then there exists only one structure $\O$ of a concrete affine algebraic equivariety on $V\xx W$ with its Zariski topology, such that $(V\xx W,\O)$ is a product of $V$ and $W$ as affine algebraic equivarieties.
\end{lem}
\begin{proof}
Let $\O$ and $\O'$ be two concrete sheaves on $V\xx W$ such that $(V\xx W,\O)$ and $(V\xx W,\O')$ are both products of $V$ and $W$ as affine algebraic equivarieties. In particular, the projections $\pi_V,\pi_W:V\xx W\to V,W$ are regular for both structures, so by universal property of $(V\xx W,\O)$, there exists a unique regular map $\theta:V\xx W\to V\xx W$ such that $\pi_V\circ \theta_V=\theta_V$ and $\pi_W\circ \theta_W=\theta_W$, and this entails that $\theta=Id$, the identity map $Id:V\xx W\to V\xx W$, which is therefore regular. Now let $U\subs V\xx W$ be open and $s\in \O'(U)$ : we get $s=s\circ Id|_U\in \O(Id^{-1}U)=\O(U)$, so $\O'(U)\subs \O(U)$; by symmetry, we have $\O(U)=\O'(U)$, so $\O=\O'$ and the concrete product structure on $V\xx W$ is unique.
\end{proof}

\section{Products of equivarieties}\label{PRODEQ}
\subsection*{Glueing concrete equivarieties}
We are going to define the product of two (concrete) equivarieties by using the products of the members of an open affine cover of each one, so we need a means to \og glue" together (concrete) equivarieties in a unique way.

\begin{lem}\label{MAG.4.13}
Let $V=\bigcup_i V_i$ be a union of concrete equivarieties over $k$, such that for all $i,j$, $V_i\cap V_j$ is open in $V_i$ and $V_j$, and $\O_{V_i}|_{V_i\cap V_j}=\O_{V_j}|_{V_i\cap V_j}$. There exists a unique structure of a concrete equivariety $\O_V$ over $k$ on $V$ such that :\\
i) every $V_i$ is open in $V$\\
ii) for each $i$, $\O_{V}|_{V_i}=\O_{V_i}$.\\
Furthermore, a regular map from $V$ into a concrete equivariety $W$ is essentially a family $\phi_i:V_i\to W$ of regular maps such that $\phi_i|_{V_i\cap V_j}=\phi_j|_{V_i\cap V_j}$ for all $i,j$.
\end{lem}
\begin{proof}
We follow \cite{MAG}, Proposition 4.13, and consider on $V$ the topology which open sets are all unions of open subsets of the $V_i$'s. For $U\subs V$ open, let $U_i=U\cap V_i$ for each $i$ and for all $i,j$, let $U_{ij}=U_i\cap U_j$ : we have $\O_{V_i}(U_{ij})=\O_{V_j}(U_{ij})$ by hypothesis, so we let $\O_V(U)$ be the set of all functions $f:U\to k$ such that for each $i$, $f|_{U_i}\in \O_{V_i}(U_i)$; it is obvious that $\O_V$ thus defined is a sheaf on $V$ for the aforementioned topology. If $U\subs V_i$ is open, it is open in $V$, and by definition of $\O_V$, we have $\O_V(U)=\O_{V_i}(U)$, so $\O_V|_{V_i}=\O_{V_i}$ and this is the only way of defining $\O_V$ with this property. Now if $P\in V$, there exists $i$ such that $P\in V_i$, and as $V_i$ is a concrete equivariety over $k$, there exists an open subset $U\subs V_i$, such that $\O_{V_i}|_{U}$ is an affine algebraic equivariety over $k$; as $U$ is open in $V$ by definition, this shows that $(V,\O_V)$ is a (concrete) equivariety over $k$. 
As for the last assertion, if $\phi_i:V_i\to W$ is a family of regular maps such that $\phi_i|_{V_i\cap V_j}=\phi_j|_{V_i\cap V_j}$ for all $i,j$ it is clear that $\phi:V\to W$, $v\in V\mapsto\phi_i(v)$ for every $i$ such that $v\in V_i$ is well defined and continuous by definition of the topology on $V$, so let us show that $\phi$ is regular : if $U\subs W$ is open and $f\in \O_W(U)$, we want to show that $f\circ \phi|_{\phi^{-1}U}\in \O_V(\phi^{-1}U)$; by definition of $\O_V$, it suffices to show that $(f\circ \phi|_{\phi^{-1}U})|_{\phi_i^{-1} U}\in \O_{V_i}(\phi_i^{-1}U)$ for all $i$. Now we have $(f\circ\phi|_{\phi^{-1}U})|_{\phi_i^{-1}U}=f\circ \phi|_{\phi_i^{-1}U}=f\circ \phi_i|_{\phi_i^{-1} U}\in \O_{V_i}(\phi_i^{-1} U)$, because $\phi_i$ is regular; it follows that $(f\circ\phi|_{\phi^{-1}U})|_{\phi_i^{-1} U}\in \O_V(\phi_i^{-1}U)$, so $\phi$ is indeed regular. By definition, $\phi$ is the unique regular map extending all the $\phi_i$'s.
\end{proof}

\begin{lem}\label{UNIGLU}
With the notations and hypotheses of Lemma \ref{MAG.4.13}, the structure on $V$ is unique, in the sense that if $(V,\O_V)$ is a concrete equivariety, $V=\bigcup_i V_i$ for equivarieties $(V_i,\O_{V_i})$ with $V_i$ open in $V$ and $\O_V|_{V_i}=\O_{V_i}$ for each $i$, then the structure induced by the $V_i$'s is the original structure on $V$.
\end{lem}
\begin{proof}
By hypothesis, for all $(i,j)$ the set $V_i\cap V_j$ is open in $V_i$, and $\O_{V_i}|_{V_i\cap V_j}=\O_V|_{V_i\cap V_j}=\O_{V_j}|_{V_i\cap V_j}$. By Lemma \ref{MAG.4.13}, there exists a unique structure $\O_W$ of a concrete equivariety over $V$ such that each $V_i$ is open in $V$ and for each $i$, $\O_W|_{V_i}=\O_{V_i}$. As by definition, $\O_V$ is such a structure, we have $\O_V=\O_W$.
\end{proof}

\subsection*{The product sheaf} Let $V,W$ be concrete equivarieties over $k$ : we may find covers $V=\bigcup_i V_i$ and $W=\bigcup_j W_j$ by open affine algebraic subvarieties. First by the end of the preceding section 
we have a concrete and quite simple description of the structural sheaves $\O_{V_i\xx W_j}$. Now for all $a,b$, $V_i\cap V_a$ is open in $V_i$ and $V_a$ and $W_j\cap W_b$ is open in $W_j$ and $W_b$ by definition of $V,W$. As the (Zariski) topology on $V_i\xx W_j$ and $V_a\xx W_b$ is finer than the product topology, it follows that $(V_i\xx W_j)\cap (V_a\xx W_b)=(V_i\cap V_a)\xx (W_j\cap W_b)$ is open in $V_i\xx W_j$ and $V_a\xx W_b$, so we may consider the application of Lemma \ref{MAG.4.13} to the definition of a structural sheaf on $V\xx W=\bigcup_{i,j} V_i\xx W_j$ (a true set-theoretic equality). Writing $X_{ij}=V_i\xx W_j$ and $X_{ab}=V_a\xx W_b$ we only need to check that $\O_{X_{ij}}|_{X_{ij}\cap X_{ab}}=\O_{X_{ab}}|_{X_{ij}\cap X_{ab}}$. For this purpose, let $O\subs X_{ij}\cap X_{ab}=(V_i\cap V_a)\xx (W_j\cap W_b)$ be an open subset and $s\in \O_{X_{ij}}(O)$ : by definition of $\O_{X_{ij}}$ each $(P,Q)\in O$ has an open neighbourhood $O'\subs O$ over which $s$ is represented as $$(P',Q')\in O'\mapsto \frac{\sum_u f_u(P')g_u(Q')}{\sum_v l_v(P')m_v(Q')}$$ with $f_u,l_v\in J(V_i)$ and $g_u,m_v\in J(W_j)$ say. We have to refine the description of the affine product sheaf on the product of two affine open subsets :

\begin{lem}\label{IPROAFF}
If $V$ and $W$ are concrete affine equivarieties and $O\subs V$, $U\subs W$ are affine open subsets, then :\\
i) the Zariski topology on $O\xx U$ is the topology induced on $O\xx U$ by the Zariski topology on $V\xx W$\\
ii) $\O_{O\xx U}=\O_{V\xx W}|_{O\xx U}$.
\end{lem}
\begin{proof}
i) Let $X\subs V\xx W$ be open for the induced topology : by definition there exists a subset $S\subs J(V)*J(W)$ such that $X=\{(P,Q)\in V\xx W : \exists f\in S,f(P,Q)\neq 0\}$, and we may assume that every $f\in S$ has the form $\sum_i g_i *h_i$, so $X\cap (O\xx U)=\{(P,Q)\in O\xx U : \exists f=\sum_i g_i*h_i\in S,\sum_i g_i|_O(P)h_i|_U(Q)\neq 0\}$, which is clearly an open subset of $O\xx U$ for the Zariski topology. Conversely, if $X\subs O\xx U$ is open, there exists a subset $S\subs J(O)*J(U)$ (for the structures induced on $O$ and $U$ by $V$ and $W$ respectively) such that $X=\{(P,Q)\in O\xx U : \exists f\in S,f(P,Q)\neq 0\}$. Let $(P,Q)\in O\xx U$ : for each $f=\sum_i g_i*h_i\in S$ (again we may assume that the members of $S$ have this form), there exist open subsets $O_{f,P}\subs O$ and $U_{f,Q}\subs U$ and $u_{i,P},m_{i,P}\in J(V)$, $v_{i,Q},n_{i,Q}\in J(W)$ such that $(P,Q)\in O_{f,P}\xx U_{f,Q}$ and for each $(P',Q')\in O_{f,P}\xx U_{f,Q}$, we have $f(P',Q')=\sum_i (u_{i,P}(P')/m_{i,P}(P')) (v_{i,Q}(Q')/n_{i,Q}(Q'))$. Reducing to the same denominator, we may rewrite this as $f(P',Q')=\sum_i (u_{i,P}(P')v_{i,Q}(Q'))/(m_{i,P}(P')n_{i,Q}(Q'))=(\sum_i \wt{u_{i,P}}(P')\wt{v_{i,Q}}(Q'))/m_P(P')n_Q(Q')$, with $m_P(P')n_Q(Q')\neq 0$. Write $g_{f,P,Q}=\sum_i \wt{u_{i,P}}*\wt{v_{i,Q}}$ the element of $J(V)*J(W)$ arising from such a choice of $O_{f,P}$, $U_{f,Q}$ and a local representation of $f$, and $X_{(f,P,Q)}=\{(P',Q')\in (O_{f,P}\cap D_O(m_P))\xx (U_{f,Q}\cap D_U(n_Q)) : g_{f,P,Q}(P',Q')\neq 0\}$ : it is now easy to see that as $X=\bigcup_{(f,P,Q)\in S\xx O\xx U} X_{f,P,Q}$, and as each $X_{f,P,Q}$ is open in $V\xx W$ for the Zariski topology, $X$ itself is open for the induced topology.\\
ii) It suffices to do the proof for $V,W$ affine subvarieties and for this, to show that $(O\xx U,\O_{V\xx W}|_{O\xx U})$ has the universal property of $(O\xx U,\O_{O\xx U})$ by Lemma \ref{UNIPROD}. First we show that the two projections $\pi_O,\pi_U:O\xx U\to O,U$ are regular for the equivariety structure $(O\xx U,\O_{V\xx W}|_{O\xx U})$, and it suffices to do it for $\pi_O$. We already know that $\pi_O$ is continuous, so let $X\subs O$ be an open subset, $s\in \O_O(X)=\O_V(X)$ and $P\in X$ : there exists an open neighbourhood $X_P\subs X$ of $P$ and $f_P,g_P\in J(V)$ such that for each $Q\in X_P$, we have $s(Q)=f_P(Q)/g_P(Q)$, and we have $X=\bigcup_{P\in X} X_P$. It follows that $\pi_O^{-1} X=\bigcup_{P\in X} \pi_O^{-1} X_P$, and for each $P\in X$, if $(P',Q')\in \pi_O^{-1} X_P$ we have $s\circ \pi_O(P',Q')=f_P(P')/g_P(P')=(f_P*1)(P',Q')/(g_P*1)(P',Q')$, which shows that $s\circ \pi_O|_{\pi_O^{-1} X}\in \O_{V\xx W}(\pi_O^{-1}X)=\O_{V\xx W}|_{O\xx U}(\pi_O^{-1} X)$, so that $\pi_O$ is indeed regular. Now let $\phi,\psi:Z\to O,U$ be two regular morphisms with $Z$ a concrete affine equivariety over $k$, say. If $O'\subs V$ is open and $s\in \O_V^\#(O')$, then $s\circ \phi|_{\phi^{-1} O'}=s|_{O\cap O'}\circ \phi|_{\phi^{-1} O'}\in \O_Z(\phi^{-1} O')$, so $\phi$ is regular as a map $Z\to V$ and likewise $\psi:Z\to W$ is regular. By universal property of $(V\xx W,\O_{V\xx W})$, there exists a unique regular map $\theta:Z\to V\xx W$ such that $\pi_V\circ \theta=\phi$ and $\pi_W\circ \theta=\psi$, with $\pi_V,\pi_W:V\xx W\to V,W$ the canonical projections. Now for each $z\in Z$, we have $\theta(z)\in O\xx U$, and $\theta:Z\to O\xx U$ is continuous (if $X\subs O\xx U$ is open, as $O\xx U$ itself is open in $V\xx W$, by (i) $X$ is open in $V\xx W$, so $\theta^{-1}(X)$ is open). If $X\subs O\xx U$ is open, and $s\in \O_{V\xx W}|_{O\xx U}(X)=\O_{V\xx W}(X)$, by regularity of $\theta$ we have $s\circ\theta|_{\theta^{-1}X}\in \O_Z(\theta^{-1} X)$, so $\theta:(Z,\O_Z)\to (O\xx U,\O_{V\xx W}|_{O\xx U})$ factorises $\phi$ and $\psi$ through $\pi_O,\pi_U:O\xx U\to O,U$, therefore $(O\xx U,\O_{V\xx W}|_{O\xx U})$ is indeed a product of $O$ and $U$ as concrete affine algebraic equivarieties. 
\end{proof}

\noindent Now $O'$ is an open subset of $V_a\cap W_b$, and we show that $s|_{O'}\in \O_{X_{ab}}(O')$. By Lemma \ref{OPENPROJ} the canonical projections $\pi_a,\pi_b:X_{ab}\to V_a,W_b$ are open, so $O'_a:=\pi_a(O')\subs V_a$ and $O'_b:=\pi_b(O')\subs W_b$ are open. Let $U_a\subs O'_a$ and $U_b\subs O'_b$ be two affine open subsets such that $(P,Q)\in U_a\xx U_b$. For every $(P',Q')\in U=O'\cap (U_a\xx U_b)$, we have $$s(P',Q')=\frac{\sum_u f_u|_{U_a}(P')g_u|_{U_b}(Q')}{\sum_v l_v|_{U_a}(P') m_v|_{U_b}(Q')},$$ the restriction to $U$ of an element of $\O_{U_a\xx U_b}(U_a\xx U_b)=\O_{V_a\xx W_b}(U_a\xx U_b)$ by Lemma \ref{IPROAFF}, and therefore every $(P,Q)\in O'$ has an open neighbourhood in $O'$ over which the restriction of $s$ is a function of the sheaf $\O_{V_a\xx W_b}$, so that $s|_{O'}\in \O_{X_{ab}}(O')$, whence by the local characterisation of sections of a sheaf, we get $s\in \O_{X_{ab}}(O)$, so $\O_{X_{ij}}(O)\subs\O_{X_{ab}}(O)$ and thus $\O_{X_{ij}}(O)=\O_{X_{ab}}(O)$ by symmetry, and therefore $\O_{X_{ij}}|_{X_{ij}\cap X_{ab}}=\O_{X_{ab}}|_{X_{ij}\cap X_{ab}}$. The conditions of Lemma \ref{MAG.4.13} apply and there exists a unique structure $\O_{V\xx W}$ of a concrete equivariety over $k$ on $V\xx W$, which restricts to the canonical structure on $V_i\xx W_j$ for all $i,j$. We now have to check that 

\begin{lem}
The product structure $\O_{V\xx W}$ does not depend on the choice of the covers $V=\bigcup_i V_i$ and $W=\bigcup_j W_j$. In other words, if $V=\bigcup_x V'_x$ and $W=\bigcup_y W'_y$ are other open covers by affine subvarieties which induce the same structures on $V$ and $W$, then the product structure defined by these covers is $\O_{V\xx W}$.
\end{lem}
\begin{proof}
By Lemma \ref{UNIGLU}, it suffices to show that for each pair $(x,y)$ :\\
i) $V'_x\xx W'_y$ is open in $V\xx W$ for the topology on $V\xx W$ induced by the original covers\\
ii) $\O_{V\xx W}|_{V'_x\xx W'_y}=\O_{V'_x\xx W'_y}$.\\
i) For all $(x,y)$ and $(i,j)$, we have $(V'_x\xx W'_y)\cap (V_i\xx W_j)=(V'_x\cap V_i)\xx (W'_y\cap W_j)$, and as $V'_x\cap V_i$ is open in $V_i$ and $W'_y\cap W_j$ is open in $W_j$, this intersection is open in $V\xx W$ by definition of the Zariski topology, so $V'_x\xx W'_y$ is open in $V\xx W$ as a union of open subsets.\\
ii) Write $Y_{xy}=V'_x\cap W'_y$, and let $U\subs Y_{xy}$ be an open subset and $s\in \O_{V\xx W}(U)$. For each pair $(i,j)$, we have $s|_{U\cap X_{ij}}\in \O_{V\xx W}(U\cap X_{ij})=\O_{V_i\xx W_j}(U\cap X_{ij})$ so by definition, for each $(P,Q)\in U\cap X_{ij}$ there is an open neighbourhood $U_{P,Q}\subs U\cap X_{ij}$ of $(P,Q)$ and some $f_a,l_b\in J(V_i)$, $g_a,m_b\in J(W_j)$, with $$s|_{U_{P,Q}}\equiv \dfrac{\sum_a f_a * g_a}{\sum_b l_b * m_b}.$$ Shrinking $U_{P,Q}$ if necessary, we may assume by Lemma \ref{OPENBASIS} that there exist affine open subsets $O^1_{P,Q}\subs V'_x\cap V_i$ and $O^2_{P,Q}\subs W'_y\cap W_j$ of $V$ and $W$ respectively, and such that $$s|_{U_{P,Q}}\equiv \dfrac{\sum_a f_a|_{O^1_{P,Q}} * g_a|_{O^2_{P,Q}}}{\sum_b l_b|_{O^1_{P,Q}} * m_b|_{O^2_{P,Q}}}$$ and by Lemma \ref{IPROAFF}, we have $s|_{U_{P,Q}}\in \O_{O^1_{P,Q}\xx O^2_{P,Q}}(U_{P,Q})$. As $V'_x$ and $W'_y$ are affine open subvarieties of $V$ and $W$ themselves, by the same lemma we have $s|_{U_{P,Q}}\in \O_{Y_{xy}}(U_{P,Q})$ for all $(P,Q)\in U\cap X_{ij}$, and as the $U_{P,Q}$'s cover $U\cap X_{ij}$, it follows that $s|_{U\cap X_{ij}}\in \O_{Y_{xy}}(U\cap X_{ij})$ for all $(i,j)$, whence finally $s\in \O_{Y_{xy}}(U)$. The proof of the converse (i.e. that $\O_{Y_{xy}}(U)\subs \O_{V\xx W}(U)$) follows the same lines, so (ii) is valid and the lemma is proved.
\end{proof}

\noindent It remains to show that with this structure, $(V\xx W,\O_{V\xx W})$ is a product of $V$ and $W$ in the category of equivarieties over $k$.

\subsection*{Projections and the product property}
We need the following two slightly different generalisations of \cite{MAG}, Corollary 2.19.

\begin{lem}\label{MAG.4.19}
If $U,V$ and $W$ are concrete affine algebraic equivarieties over $k$ and $\phi:U\to V\xx W$ is a map, then $\phi$ is regular if and only if $p\circ\phi$ and $q\circ \phi$ are regular, for $p:V\xx W\to V$ and $q:V\xx W\to W$ the canonical projection maps.
\end{lem}
\begin{proof}
As $p$ and $q$ are regular by Lemma \ref{OPENPROJ}, if $\phi$ is regular then $p\circ \phi$ and $q\circ\phi$ are regular by composition. Conversely, if $p\circ\phi:U\to V$ and $q\circ\phi:U\to W$ are regular, as $(V\xx W,\O_{V\xx W})$ is a product by Proposition \ref{AFFPROD}, there exists a unique regular map $\theta:U\to V\xx W$ such that $p\circ\theta=p\circ\phi$ and $q\circ\theta=q\circ\phi$. Now the underlying set of $V\xx W$ is the set-theoretic cartesian product of $V$ and $W$, so $\theta=\phi$, and therefore $\phi$ is regular.
\end{proof}

\noindent For all $i,j$, consider the projection maps $p_i,q_j:V_i\xx W_j\to V,W$ : as $V_i$ is open in $V$ and $W_j$ is open in $W$, they are obviously regular, so by Lemma \ref{MAG.4.13} again, the projection maps $p,q:V\xx W\to V,W$, obtained by glueing, are regular as well. It is therefore possible to generalise the preceding lemma to

\begin{lem}\label{REGPROD}
If $U,V$ and $W$ are concrete equivarieties and $\phi:U\to V\xx W$ is a map, then $\phi$ is regular if and only if $p\circ \phi$ and $q\circ\phi$ are regular.
\end{lem}
\begin{proof}
It suffices to show, as in Lemma \ref{MAG.4.19}, that if $p\circ\phi$ and $q\circ\phi$ are regular, then $\phi$ is regular. In this case, write $V\xx W=\bigcup_{ij} (V_i\xx W_j)$ as before, and let $U_{ij}:=\phi^{-1} (V_i\xx W_j)$ for all $(i,j)$ : as $U_{ij}=(p\phi)^{-1}(V_i)\cap (q\phi)^{-1}(W_j)$ and $p\phi$ and $q\phi$ are continuous by hypothesis, $U_{ij}$ is open, so $U=\bigcup_{ij} U_{ij}$ is an open cover of $U$. Choose a pair $(i,j)$ : as $U_{ij}$ is open, it is a union $U_{ij}=\bigcup_l U_l$ of affine open subsets by Lemma \ref{OPENBASIS}, and we let $\phi_l:=\phi|_{U_l}=\phi_{ij}|_{U_l}$ for each $l$, where $\phi_{ij}=\phi|_{U_{ij}}$. By hypothesis, for each $l$ the maps $p_{ij}\circ \phi_l$ and $q_{ij}\circ\phi_l$ are regular as restrictions of $p\circ \phi$ and $q\circ\phi$ to $U_l$, respectively, so by Lemma \ref{MAG.4.19} (the affine case), $\phi_l$ is regular. Now the restriction of $\phi_{ij}$ to each $U_l$ is regular, so $\phi_{ij}$ is regular by the last part of Lemma \ref{MAG.4.13}, and for the same reason, $\phi$ is regular.
\end{proof}

\begin{prop}\label{MAG.4.21}
If $V$ and $W$ are concrete equivarieties over $k$, then $(V\xx W,\O_{V\xx W})$ is a product of $V$ and $W$ in the category of equivarieties over $k$.
\end{prop}
\begin{proof}
Let $U$ be a concrete equivariety over $k$, and $\phi:U\to V$, $\psi:U\to W$ two regular maps. Consider the product map $\theta:U\to V\xx W$, and the projection maps $p:V\xx W\to V$, $q:V\xx W\to W$, which are regular by what precedes : by definition of $\theta$, we have $p\circ\theta=\phi$ and $q\circ\theta=\psi$, so by Lemma \ref{REGPROD}, the map $\theta$ is regular. It is the only possible map such that $p\circ\theta=\phi$ and $q\circ\theta=\psi$, so $(V\xx W,\O_{V\xx W})$ is indeed a product of $V$ and $W$ in the category of equivarieties over $k$.
\end{proof}

\begin{cor}\label{REGPROD'}
If $\phi:V\to W$ and $\psi:X\to Y$ are regular maps of concrete equivarities, then the product map $\phi\xx \psi:V\xx X\to W\xx Y$ is regular. 
\end{cor}
\begin{proof}
Let $p,q:V\xx X\to V,X$ and $\pi,\rho:W\xx Y\to W,Y$ be the canonical projection maps of the two products. By composition, the maps $\phi\circ p:V\xx X\to V\to W$ and $\psi\circ q:V\xx X\to X\to Y$ are regular, and by Proposition \ref{MAG.4.21} there exists a unique regular map $\theta:V\xx X\to W\xx Y$ such that $\pi\circ \theta=\phi\circ p$ and $\rho\circ\theta=\psi\circ q$. In particular, for every $(P,Q)\in V\xx X$ we have $\theta(P,Q)=(\phi(P),\psi(Q))$, so $\theta=\phi\xx\psi$, which is therefore regular. 
\end{proof}

\section{Subvarieties of concrete equivarieties}\label{SUBVAR}
\subsection*{Immersions and subvarieties}
If $(V,\O_V)$ is a concrete equivariety, we may generalise the restriction of $\O_V$ to open subsets of $V$, as follows. If $S\subs V$ is any subset, define the \emph{restriction of $\O_V$ to $S$}, noted $\O_V|_S=\O_S$, as the following sheaf : if $U\subs S$ is an open subset for the induced topology, we let $\O_S(U)=\{f:U\to k : \forall P\in U, \exists O\ni P, O$ open in $V, \exists g\in \O_V(O),\ f|_{O\cap U}=g|_{O\cap U}\}$ (in other words, $\O_S(U)$ is the set of functions $U\to k$ which are \emph{locally} the restriction of a function of $\O_V$). It is obviously a subpresheaf of the sheaf of $k$-valued functions, and if $U=\bigcup_I U_i$ is an open cover in $S$ and $(f_i\in \O_S(U_i) : i\in I)$ is a family of compatible sections, they together define a function $f:U\to k$; now if $P\in U$, say $P\in U_i$, there exist an ambiant open $O\ni P$ and $g\in \O_V(O)$ such that $f|_{O\cap U_i}=f_i|_{O\cap U_i}=g|_{O\cap U_i}$, and if we write $U_i=O_i\cap U$ for some open $O_i\subs V$, $W=O\cap O_i$ and $h=g|_{W}\in \O_V(W)$, we have $P\in W$ and we get $h|_{U\cap W}=g|_{O\cap U_i}=f|_{O\cap U_i}=f|_{U\cap W}$, whence $f\in \O_S(U)$ and $\O_S$ is indeed a sheaf.
\begin{rem}
This is the usual notion of the restriction of a sheaf (see for instance \cite{HAG}, II.1 for the general definition); working with concrete equivarieties simplifies its description, which will be useful for the study of open, closed, and locally closed subvarieties.
\end{rem}

\noindent Recall that a subset $S$ of a topological space $X$ is \emph{locally closed}, if it is the intersection of an open and a closed subset of $X$, or equivalently if every point $P\in S$ has an open neighbourhood $O$ in $X$ such that $O\cap S$ is (relatively) closed in $O$.

\begin{defi}
Let $V$ and $W$ be equivarieties.\\
i) A \emph{subvariety of $V$} is a subset $S$ together with the restricted sheaf $\O_S:=\O_V|_S$, and such that $(S,\O_S)$ is an equivariety.\\
ii) An \emph{immersion} of $V$ into $W$ is a regular map $i:V\into W$, such that $i(V)$ is a subvariety and $i$ is an isomorphism between $V$ and $(i(V),\O_W|_{i(V)})$. We say that $i$ is \emph{open} if $i(V)$ is, and that $i$ is \emph{closed}, if $i(V)$ is.
\end{defi}

\begin{ex}\label{AFFCLO}
Here is a fundamental example of a closed immersion. If $A$ is an affine $*$-algebra and $I$ is an ideal of $A$, we have the map $\phi:Y:=K(A/I) \to X:=K(A)$ (for $K$ the functor $Spm$ as in \cite{EQAG}, Section 4) defined as $\phi:=K(\pi)$, for $\pi:A\onto A/I$ the canonical projection. The corestriction $\psi$ of $\phi$ to its image is clearly a homeomorphism and using the standard definition of the structural sheaves on $X$ and $Y$, it is clear that $(\psi,\psi^\#)$ is an isomorphism, for $\psi^\#:\O_{\phi(Y)}\to \psi_*\O_Y$ defined as in \cite{EQAG}.
\end{ex}

\noindent Whereas an open subset of an equivariety obviously inherits the structure of an equivariety (and so is a subvariety in our sense), it is not obvious from the definitions that every closed subset inherits such a structure.

\subsection*{Locally closed subvarieties}
In this section we will show that every locally closed subset of an equivariety is itself an equivariety, i.e. a subvariety.

\begin{lem}\label{CLSVAR}
If $(V,\O_V)$ is a concrete affine equivariety and $X\subs V$ is a closed subset, then $(X,\O_V|_X)$ is an affine equivariety and the restriction $\phi|_X$ of every chart $\phi:V\cong W\subs k^n$ to $(X,\O_V|_X)$ is itself a regular isomorphism.
\end{lem}
\begin{proof}
First we show that if $V\subs k^n$ is an affine subvariety, then $(X,\O_V|_X)$ is an affine equivariety, by proving that $\O_V|_X$ is the structural sheaf of $X$ as considered itself as a subvariety of $k^n$. Let $U\subs X$ be an open subset, and $\O_X$ the structural sheaf of $X$ as considered as a subvariety of $k^n$. If $f\in \O_X(U)$, for each $P\in U$ there exist an open neighbourhood $U_P\subs U$ of $P$ in $X$ and $G,H\in k[\ov X]$ such that $U_P\subs D_X(H)$ and $f|_{U_P}\equiv G/H$. Let $O=D_V(H)$, an open subset of $V$ : on $O$, $G/H$ defines a regular function $g\in \O_V(O)$ say, and we have $g|_{O\cap U_P}:Q\in O\cap U_P\mapsto G(Q)/H(Q)=f(Q)$, so $g|_{O\cap U_P}=f|_{O\cap U_P}$, whence $f|_{U_P}\in \O_V|_X(U_P)$ and by the local character of a sheaf, $f\in \O_V|_X(U)$ : we have $\O_X(U)\subs \O_V|_X(U)$. The other way round, if $f\in \O_V|_X(U)$, for each $P\in U$ let $O_P\subs V$ be an open neighbourhood of $P$ in $V$ and $g\in \O_V(O_P)$ such that $g|_{O_P\cap U}=f|_{O_P\cap U}$ : shrinking $O_P$ if necessary, there exist $G,H\in k[\ov X]$ such that $g|_{O_P}\equiv G/H$; now $O_P\cap U$ is open in $X$, and for each $Q\in O_P\cap U$ we have $f(Q)=g(Q)=G(Q)/h(Q)$, so that $f|_{O_P\cap U}\equiv G/H$, which shows that $f\in \O_X(U)$, and thus $\O_X(U)=\O_V|_X(U)$ and the two sheaves coincide.
Now let $\phi:V\cong W\subs k^n$ be a regular isomorphism with a subvariety, and write $\O_X=\O_V|_X$ as before, and $\phi(X):=Y\subs W$, a closed subset : by what precedes we have $\O_Y=\O_W|_Y$. Let $U\subs Y$ be an open subset : if $s\in \O_Y(U)$ and $P\in \phi^{-1}U$ then $\phi(P)=Q\in U$ and there exists an open neighbourhood $U_Q\subs U$ of $Q$ in $Y$ and $G,H\in k[\ov X]$ such that $s|_{U_Q}\equiv G/H\equiv g/h$, where $g=G+\ms I(W)$ and $h=H+\ms I(W)$. For each $P'\in \phi^{-1} U_Q$, we have $s\circ \phi(P')=s(\phi(P'))=g(\phi(P'))/h(\phi(P'))=(g/h)\circ \phi|_{\phi^{-1}U_Q} (P')$, so $s\circ \phi|_{\phi^{-1}U_Q}=(g/h\circ \phi|_{\phi^{-1}U_Q})$, and as $\phi$ is regular, so that $(g/h)\circ \phi|_{\phi^{-1}U_Q}\in \O_V(\phi^{-1} U_Q)$, this means that $s\circ \phi|_{\phi^{-1} U}\in \O_X(\phi^{-1} U)$ by definition of $\O_V|_X$ and because the $\phi^{-1}(U_Q)$'s cover $\phi^{-1}U$, and all this shows that $\phi|_X:X\cong Y$ is a regular map. Conversely, if $t\in \O_X(U)$, where $U\subs X$ is open, let $Q\in \phi(U)$ and $P=\phi^{-1}Q$ : by definition of the restriction $\O_X$, there exist an open neighbourhood $O$ of $P$ in $V$ and $f\in \O_V(O)$ such that $f|_{U\cap O}=t|_{U\cap O}$; let $O'=\phi(O)$, an open neighbourhood of $Q$ in $W$ : we have $f\circ \phi^{-1}|_{O'}\in \O_V(O')$, because $\phi^{-1}$ is regular, and $(f\circ \phi^{-1}|_{O'})|_{\phi(U)\cap O'}=(t\circ \phi^{-1}|_{O'})|_{\phi(U)\cap O'}$. By definition, this means that $t\circ \phi^{-1}|_{\phi(U)}\in \O_W|_Y(\phi(U))=\O_Y(\phi(U))$, whence $(\phi|_X)^{-1}$ is regular as well, and $\phi|_X:X\cong Y$ is an isomorphism. In particular, $(X,\O_X)$ is an affine algebraic equivariety.
\end{proof}

\begin{lem}\label{DBRES}
If $(V,\O_V)$ is a concrete equivariety and $S\subs T\subs V$ are subsets, then we have $(\O_V|_T)|_S=\O_V|_S$.
\end{lem}
\begin{proof}
Let $U\subs S$ be open, and let $s\in \O_V|_S(U)$ : by definition, for each $P\in U$ there exists $O_P\subs V$ open and $t_P\in \O_V(O_P)$ such that $t_P|_{O_P\cap U}=s|_{O_P\cap U}$. Let $O'_P=O_P\cap T$ for each $P$ : $O'_P$ is open in $T$ and by definition of $\O_V|_T$, we have $t'_P:=t_P|_{O'_P}\in \O_V|_T(O'_P)$ : as $t'_P|_{O'_P\cap U}=s|_{O'_P\cap U}$ for each $P$, this show that $s\in (\O_V|_T)|_S(U)$. Conversely, if $s\in (\O_V|_T)|_S(U)$, for each $P\in U$ there exists an open $O'_P=O_P\cap T$ of $T$ with $O_P$ open in $V$, and $t'_P\in \O_V|_T(O'_P)$ such that $t'_P|_{O'_P\cap U}=s|_{O'_P\cap U}$. By definition of $\O_V|_T$, there exists an open $U_P\subs V$ and $t_P\in \O_V(U_P)$ such that $P\in U_P$ and $t_P|_{U_P\cap O'_P}=t'_P|_{U_P\cap O'_P}$ : we have $t_P|_{(U_P\cap O_P)\cap U}=t_P|_{U_P\cap O'_P\cap U}=t'_P|_{U_P\cap (O'_P\cap U)}=s|_{U_P\cap (O'_P\cap U)}=s|_{(U_P\cap O_P)\cap U}$, so that $s\in \O_V|_S(U)$, $(\O_V|_T)|_S(U)=\O_V|_S(U)$, and therefore $(\O_V|_T)|_S=\O_V|_S$, and the proof is complete.
\end{proof}

\begin{prop}\label{LOCLOS}
Every locally closed subset $S$ of a (concrete) equivariety $V$, equipped with the restriction of the structure sheaf, is an equivariety itself, and the inclusion $i:S\into V$ is a regular map (and therefore an immersion).
\end{prop}
\begin{proof}
Let $(V,\O_V)$ be a concrete equivariety. First, if $O\subs V$ is an open subvariety, then by Lemma \ref{OPENBASIS} $O$ is the union of affine open subsets of $V$, so $(O,\O_O)$ itself is an equivariety. Secondly, if $X\subs V$ is a closed subvariety, write $V=\bigcup_i V_i$ with each $V_i$ an affine open subvariety of $V$ : we have $X=\bigcup_i X_i$, with $X_i=X\cap V_i$ for each $i$. Now $X_i$ is closed in $V_i$ because $X$ is closed in $V$, so by Lemma \ref{CLSVAR}, $(X_i,\O_{V_i}|_{X_i})$ is an affine equivariety. Furthermore, as now $V_i$ is open in $V$ for each $i$, $X_i$ is open in $X$ for each $i$, and as $(X_i,\O_{V_i}|_{X_i})=(X_i,(\O_V|_X)|_{X_i})$ for each $i$ by Lemma \ref{DBRES}, $X$ is a union of affine algebraic equivarieties, each open in $X$, i.e. $X$ is an equivariety.
In the general case, if $S\subs V$ is locally closed, then it is the intersection $S=O\cap X$ of an open subset $O$ and a closed subset $X$ of $V$. By Lemma \ref{DBRES} again, we have $(S,\O_V|_S)=(S,((\O_V)|_O)|_S)$ and by what precedes, it is an equivariety, because $S$ is closed in $O$.
Let $i:S\into V$ be the inclusion : it is continuous for the induced topology on $S$, so let $U\subs V$ be an open subset and $f\in \O_V(U)$. For each $P\in U\cap S$, we have $P\in U$ and $(f\circ i|_ {U\cap S})|_{(U\cap S)\cap U}=f\circ i|_{U\cap S}=f|_{U\cap S}=f|_{(U\cap S)\cap U}$, which shows that $f\circ i|_{U\cap S}\in \O_S(U\cap S)$, and therefore $i$ is regular.
\end{proof}

\section{Separated and algebraic equivarieties}\label{SEPALG}
\subsection*{Separatedness and the diagonal} We introduce the notion of separatedness for equivarieties in general. Following the approach of \cite{MAG}, Chapter 4, as before, we are going to usefully characterise the separatedness property by the closed character of the diagonal as in the algebraically closed case, thanks to the results of the preceding sections.
\begin{defi}
Say that an equivariety $V$ over $k$ is :\\
i) \emph{separated}, if for every pair of morphisms $\phi_1,\phi_2:U\to V$ with $U$ an affine algebraic equivariety over $k$, the set $\{P\in U : \phi_1(P)=\phi_2(P)\}$ is closed in $U$\\
ii) \emph{algebraic}, if it is \emph{separated} and \emph{compact}.
\end{defi}

\noindent Recall from Section \ref{OPAFF} that if $V$ is a concrete affine equivariety, we have the open and regular projections (for the product structure on $V\xx V$) $p,q:V\xx V\to V$ on each component and the \emph{diagonal} $\Delta_V=\{(P,P)\in V\xx V\}=\{(P,Q)\in V\xx V : P=Q\}=\{(P,Q)\in V\xx V : \forall f\in J(V),\ f(P)=f(Q)\}$, a closed subset of $V\xx V$.
In general, if $V$ is a concrete equivariety over $k$, say $V=\bigcup_i V_i$ with each $V_i$ open and affine, for each $i$ we have the diagonal $\Delta_i=\Delta_{V_i}$ of $V_i$ defined as above, a closed subset of $V_i\xx V_i$ for the induced topology (the Zariski topology on $V_i\xx V_i$ is \emph{by definition} the induced topology). For the next proposition, recall that if $X$ is any topological space and $S$ is any subset of $X$, for any open cover $X=\bigcup_i U_i$ the subset $S$ is closed if and only if $S\cap U_i$ is closed in $U_i$ for each $i$.

\begin{prop}\label{MAG.4.24}
An equivariety $(V,\O_V)$ is separated if and only if the diagonal $\Delta=\{(P,P)\in V\xx V : P\in V\}$ of $V$ is closed in $V\xx V$ for the topology associated to the product structure on $V\xx V$.
\end{prop}
\begin{proof}
As before, we may assume by Lemma \ref{CONCEQ} that all equivarieties are concrete. As in \cite{MAG}, Proposition 4.24, assume $\Delta$ is closed, and let $\phi_1,\phi_2:U\to V$ be two regular morphisms defined on a concrete affine algebraic equivariety $U$. The product map $(\phi_1,\phi_2):U\to V\xx V$ is regular by Proposition \ref{MAG.4.21}, so in particular it is continuous and thus $\{P\in U :\phi_1(P)=\phi_2(P)\}=(\phi_1,\phi_2)^{-1}\Delta$ is closed, so $V$ is separated. Conversely, assume that $V$ is separated, and let $\phi_1,\phi_2:V\xx V\to V$ be the two projection maps, and write $V=\bigcup_i V_i$ as a union of open affine algebraic subvarieties. For each pair $(i,j)$, the restrictions $\phi_{1,ij}$ and $\phi_{2,ij}$ of $\phi_1$ and $\phi_2$ to the open subset $V_i\xx V_j$ of $V^2$ are regular by composition of $\phi_1$ and $\phi_2$ with the open immersion $V_i\xx V_j\into V^2$, and as $V$ is separated and $V_i\xx V_j$ is affine, the set $\Delta_{ij}:= \Delta\cap (V_i\xx V_j)=\{(P,Q)\in V_i\xx V_j : P=\phi_1(P,Q)=\phi_2(P,Q)=Q\}$ is closed in $V_i\xx V_j$. Now as $\Delta=\bigcup_{(i,j)} \Delta_{ij}$ and $V^2=\bigcup_{(i,j)} V_i\xx V_j$ is an open cover, it follows that $\Delta$ is closed in $V^2$.
\end{proof}

\subsection*{A local characterisation of separatedness}
\begin{lem}\label{MAG.3.22}
Let $\phi:V\to W$ be a regular morphism of concrete affine equivarieties over $k$ and $f=\phi^*:J(W)\to J(V)$, $g\mapsto g\circ \phi$ the induced $k$-morphism. The morphism $\phi$ is a closed immersion (i.e is injective with a closed image) if and only if $f$ is surjective.
\end{lem}
\begin{proof}
Suppose that $\phi$ is a closed immersion : $\phi(V)$ is closed and $\phi$ is an isomorphism onto $\phi(V)$ considered as a closed subvariety of $W$; the corestriction $\psi$ of $\phi$ to $\phi(V)$ induces an isomorphism $\psi^\#_{\phi(V)}:\O_{\phi(V)}(\phi(V))\cong \O_V(V)$, $g\mapsto g\circ \phi$. Now consider the inclusion $j:i(V)\into W$ : it induces the restriction $j_W^\#:g\in J(W)\mapsto g|_{\phi(V)}\in J(\phi(V))=\O_{\phi(V)}(\phi(V))$, and by Lemma \ref{CLSVAR}, using a chart for $W$ we see that $j_W^\#$ is surjective. As $f=\psi^\#_{\phi(V)}\circ j_W^\#$, $f$ itself is surjective. 
Conversely, suppose $f$ is surjective and factor it out through the quotient isomorphism $f_I:J(W)/I\cong J(V)$, $g+I\mapsto g\circ \phi$ for $I=Ker(f)=\{g\in J(W) : g\circ \phi\equiv 0\}$. Write $A=J(W)$, $B=J(V)$ : by Example \ref{AFFCLO}, we have a closed immersion $K(\pi):K(A/I)\into K(A)$, and by what precedes an isomorphism $K(f_I):K(B)\cong K(A/I)$, so that $K(\pi)\circ K(f_I)=K(f)$. It follows that $K(f)$ is a closed immersion, as well as $\phi$, because of the following commutative diagram :
$$\begin{CD}
K(B)=KJ(V) @>K(f)=KJ(\phi)>> KJ(W)=K(A)\\
@A\cong_V AA @AA\cong_W A\\
V @>>\phi> W
\end{CD}$$
where $\cong_V$ and $\cong_W$ come from the natural isomorphism of \cite{EQAG}, Proposition 4.10.
\end{proof}

\noindent We will need the following generalisation of Lemma \ref{IPROAFF}.
\begin{lem}\label{ZARIND}
If $V$ and $W$ are concrete equivarieties over $k$ and $O\subs V$, $U\subs W$ are affine open subsets, then the Zariski topology on $O\xx U$ is the topology induced by the Zariski topology on $V\xx W$.
\end{lem}
\begin{proof}
Write $V=\bigcup_i V_i$ and $W=\bigcup_j W_j$ as unions of open affine subsets, $O_i=O\cap V_i$ and $U_j=U\cap W_j$ for all $i,j$. Each $O_i$ and each $U_j$ is an open subset of an affine, so for all $i,j$ we may write $O_i=\bigcup_{a\in A_i} O_{i,a}$ and $U_j=\bigcup_{b\in B_j} U_{j,b}$ as finite unions of affine open subsets of $V_i$ and $W_j$ respectively. If $X\subs O\xx U$ is Zariski open, then for all $i,j,a,b$, $X\cap (O_{i,a}\xx U_{j,b})$ is Zariski open in $O_{i,a}\xx U_{j,b}$ by Lemma \ref{IPROAFF}, and therefore open in $V_i\xx W_j$ by the same Lemma, hence $X$ is open in $O\xx U$ for the topology induced by $V\xx W$. Conversely, if $X=Y\cap (O\xx U)$ is open for the topology induced by $V\xx W$, where $Y\subs V\xx W$ is open, then for all $i,j,a,b$, as $O_{i,a}\xx U_{j,b}$ is open in $V\xx W$, the set $X\cap (O_{i,a}\xx U_{j,b})=Y\cap (O_{i,a}\xx U_{j,b})$ is Zariski open in $V_i\xx W_j$; by Lemma \ref{IPROAFF}, this set is Zariski open in $O_{i,a}\xx U_{j,b}$, and by the same lemma again it is Zariski open in $O\xx U$. It follows that the two topologies on $O\xx U$ coincide.
\end{proof}

\begin{lem}\label{MAG.4.25}
If $V$ is a concrete equivariety over $k$, then the diagonal $\Delta_V$ is locally closed in $V\xx V$.
\end{lem}
\begin{proof}
The proof is essentially the same as in \cite{MAG}, Corollary 4.25. If $P\in V$, let $U$ be an affine open neighbourhood of $P$ in $V$ : $U\xx U$ is an open neighbourhood of $(P,P)$ in $V\xx V$ because the Zariski topology on $V\xx V$ is finer than the product topology. Let $\Delta_U=\Delta_V\cap (U\xx U)$ be the diagonal of $U$; as $U$ is affine, $\Delta_U$ is Zariski closed in $U\xx U$; by Lemma \ref{ZARIND}, it is closed in $U^2$ for the induced topology, so every $(P,P)\in \Delta_V$ has an open neighbourhood $U^2$ in $V^2$ such that $\Delta_V\cap U^2$ is closed in $U^2$ : by characterisation, $\Delta_V$ is locally closed.
\end{proof}

\begin{lem}\label{MAG.4.26}
If $\phi:V\to W$ is a regular map of concrete equivarieties over $k$, then the graph $\Gamma_\phi$ of $\phi$ is a locally closed subvariety of $V\xx W$, and $\Gamma_\phi$ is closed if $W$ is separated. Furthermore, the map $i:V\to \Gamma_\phi$, $P\mapsto (P,\phi(P))$ is a regular isomorphism.
\end{lem}
\begin{proof}
We generalise and expand the proof of \cite{MAG}, Corollary 4.26. By Corollary \ref{REGPROD'}, the map $(P,Q)\in V\xx W\mapsto (\phi(P),Q)\in W\xx W$ is regular as a product of the two regular maps $\phi$ and $Id_W$. As $\Gamma_\phi=(\phi\xx Id_W)^{-1} \Delta_W$, $\phi\xx Id_W$ is continuous and $\Delta_W$ is locally closed by Lemma \ref{MAG.4.25}, $\Gamma_\phi$ itself is locally closed and a subvariety by Proposition \ref{LOCLOS}. If $W$ is separated, by Proposition \ref{MAG.4.24} $\Delta_W$ is closed in $W\xx W$ and $\Gamma_\phi$ is closed.
As for the second assertion, let $i:V\to V\xx W$ be the product of $Id_V$ and $\phi$ : it is regular by Proposition \ref{MAG.4.21} and an embedding. Let $j:\Gamma_\phi\to V$ be the first projection : we have $i\circ j=Id$ and $j\circ i=Id$, and if $U\subs V$ is open, then $j^{-1} U=(U\xx W)\cap \Gamma_\phi$ is open in $\Gamma_\phi$, so $j$ is continuous; for every $s\in \O_V(U)$, we have $s\circ \pi\in \O_{V\xx W}(U\xx W)$ because $\pi:V\xx W\to V$ is regular and as for every $(P,Q)\in j^{-1} U$, we have $s\circ j(P,Q)=s(P)=s\circ \pi(P,Q)$, we get $s\circ j\in \O_{\Gamma_\phi}(j^{-1} U)$, $j$ itself is regular, therefore the corestriction of $i$ to $\Gamma_\phi$ is a regular isomorphism.
\end{proof}

\begin{prop}\label{MAG.4.27}
Let $V$ be an equivariety over $k$. The following are equivalent :\\
i) $V$ is separated\\
ii) for all affine open subvarieties $U,U'\subs V$, $U\cap U'$ is an affine open subvariety and the map $k\{U\} * k\{U'\}\to k\{U\cap U'\}$, $f*g\mapsto f|_{U\cap U'}.g|_{U\cap U'}$, is surjective\\
iii) there exists an affine open cover of $V$ such that property (ii) holds for all pairs of members of the cover (equivalently, this holds for any cover).
\end{prop}
\begin{proof}
We again follow the lines of \cite{MAG}, Theorem 4.27, replacing tensor products by $*$-products. First let $U,U'\subs V$ be two affine open subvarieties and $\Delta=\Delta_V\subs V^2$ the diagonal of $V$.\\
a) Suppose that $V$ is separated, so that $\Delta$ is closed by Proposition \ref{MAG.4.24} : the graph $\Gamma_i$ of the inclusion map $i:U\cap U'\into V$ is the set $(U\xx U')\cap \Delta \subs (U\cap U')\xx V$, a Zariski closed subset of $U\xx U'$ by Lemma \ref{ZARIND}, so by Lemma \ref{CLSVAR}, $\Gamma_i$ is an affine algebraic equivariety, as well as $U\cap U'$ by Lemma \ref{MAG.4.26} : $U\cap U'$ is an affine open subset of $V$.\\
b) Suppose $U\cap U'$ is affine. The $k$-morphism $k\{U\}*k\{U'\}\to k\{U\cap U'\}$, $f*g\mapsto f|_{U\cap U'}.g|_{U\cap U'}$ induced by the regular embedding $i:U\cap U'\into U\xx U'$ is isomorphic to the $k$-morphism $J(i):k\{U\xx U'\}\to k\{U\cap U'\}$ because $k\{U\xx U'\}\cong k\{U\} * k\{U'\}$ by the duality between affine algebraic equivarieties and affine $*$-algebras of \cite{EQAG}, Theorem 4.15. It follows that the set $(U\xx U')\cap \Delta$, which is the image of $i:U\cap U'\to U\xx U'$, is Zariski closed if and only if $k\{U\}*k\{U'\}\to k\{U\cap U'\}$ is surjective, by Lemma \ref{MAG.3.22}.\\
Now assume (i) and let $U,U'$ be as before. As $V$ is separated, by (a) the open set $U\cap U'$ is affine, and as $\Delta$ is closed in $V^2$ by Proposition \ref{MAG.4.24}, $(U\xx U')\cap \Delta$ is relatively closed in $U\xx U'$, so Zariski closed by Lemma \ref{ZARIND}; by (b), $k\{U\}*k\{U'\}\to k\{U\cap U'\}$ is surjective. 
If (ii) holds, then by definition there exists an open cover $V=\bigcup_i U_i$ for some $U_i$ open affine subvarieties, and (iii) holds.
Assume (iii), and let $V=\bigcup_i U_i$ be an open cover by affine subvarieties with this property. 
As $U_i\cap U_j$ is affine for all $i,j$, by (b) every set $(U_i\xx U_j)\cap \Delta$ is closed in $U_i\xx U_j$. It follows as in Proposition \ref{MAG.4.24} that $\Delta=\bigcup_{i,j} (U_i\xx U_j)\cap \Delta$ is closed, because $V=\bigcup_{i,j} U_i\xx U_j$ is an open cover of $V\xx V$, and thus $V$ is separated by this same proposition.
\end{proof}

\subsection*{Quasi-affine algebraic equivarieties}
\noindent Our first examples of algebraic equivarieties which are not, in general, affine algebraic equivarieties, are the quasi-affine algebraic equivarieties, i.e. the open subvarieties of affine algebraic equivarieties.

\begin{prop}\label{QAFFALGVAR}
Every quasi-affine algebraic equivariety is an algebraic equivariety.
\end{prop}
\begin{proof}
Let $W$ be a quasi-affine algebraic variety : it is an open subvariety of an affine algebraic variety $V$. 
We may assume that $V$ is an affine algebraic subvariety of $k^n$ say, and by definition of the Zariski topology we may write $W=\bigcup_{i=1}^m W_i$ for some basic open subsets $W_i=D_V(h_i)$ of $V$, with $h_i\in k[V]$ : as the $W_i$'s are affine algebraic equivarieties by \cite{EQAG}, Lemma 3.15, open in $W$ and in finite number, $W$ is a compact equivariety over $k$ by \cite{EQAG}, Lemma 3.16. Now let $U$ 
be a concrete affine algebraic equivariety and $\phi_1,\phi_2:U\to W$ regular morphisms : the embedding $i:W\into V$ is regular by Proposition \ref{LOCLOS}, so that $i\phi_1,i\phi_2:U\to V$ are regular as well; now we have $\{P\in U : \phi_1(P)=\phi_2(P)\}=\{P\in U : i\phi_1(P)=i\phi_2(P)\}$ and this set is closed because $V$ is separated; it follows that $W$ itself is separated, and is therefore an algebraic variety. 
\end{proof}
\begin{rem}
If $k$ is not algebraically closed, every quasi-affine algebraic equivariety is in fact an affine algebraic equivariety. Indeed, with the notations of the preceding proof, let $N(X_i:i\in I)$ be an appropriate normic form over $k$ : we have $W=\bigcup_i D_V(h_i)=D_V(N(h_i:i\in I))$, so $W$ is a basic open subset of $V$, which is affine by (already cited) \cite{EQAG}, Lemma 3.15. The distinction thus only occurs over algebraically closed fields; we however develop a uniform theory which is valid over any field.
\end{rem}

\section{Locally closed subvarieties of algebraic equivarieties}\label{LOCVAR}
Our original motivation for the introduction of algebraic equivarieties was to provide an intrinsic notion of a geometric space over $k$, encompassing all the usual, i.e. quasi-projective, (equi)varieties. As these are the open subspaces of projective varieties, which are themselves essentially the closed subvarieties of projective spaces, we have to deal by Theorem \ref{PROJSP} with \emph{locally closed subvarieties} of certain algebraic equivarieties. We finish this article by focusing on these and applying our results to quasi-projective equivarieties, proving that every locally closed subvariety of an algebraic equivariety is naturally an algebraic equivariety for the induced structure (Theorem \ref{LOCSEP}), which will be a sophisticated version of Proposition \ref{LOCLOS}. We begin by the description of projective spaces as algebraic equivarieties. As before, we may restrict ourselves to concrete equivarieties.

\subsection*{Projective spaces as algebraic equivarieties}
If $k$ is any field, recall that $\mbb P^n(k)$, the $n$-th projective space over $k$, is covered by the subsets $U_i=\{[a_0:\ldots:a_n] : a_i\neq 0\}$, $i=0,\ldots,n$; the \emph{Zariski topology on $\mbb P^n(k)$} is defined as in the case where $k$ is algebraically closed, i.e. the basic open subsets have the form $D(F)=\{[a_0:\ldots:a_n]\in \mbb P^n(k) : F(a_0,\ldots,a_n)\neq 0\}$ for $F$ any \emph{homogeneous} polynomial of $k[X_0,\ldots,X_n]$, and we have the bijections $u_i:U_i=D(X_i)\to \mbb A^n(k)=k^n$, $[a_0:\ldots:a_n]\mapsto (a_0/a_i,\ldots,a_n/a_i)$ ($a_i/a_i$ omitted on the right) with inverse bijections $(a_1,\ldots,a_n)\mapsto [a_1:\ldots:1:\ldots:a_n]$ ($1$ in $i$-th position on the right).
In general we have the
\begin{prop}\label{MAG.6.4}
The bijections $u_i:U_i\to k^n$ are homeomorphisms for the Zariski topology, and the $U_i$'s are open.
\end{prop}
\begin{proof}
As in the algebraically closed version (\cite{MAG}, Proposition 6.4), it suffices to prove this for $i=0$. We have $U_0=D(X_0)$, open by definition, and if $F\in k[X_0,\ldots,X_n]$ is homogeneous, $D(F)\cap U_0=\{[a_0:\ldots:a_n]\in \mbb P^n(k):F(\ov a)\neq 0\ \&\ a_0\neq 0\}=\{[\ov a]\in U_0 : F(1,a_1/a_0,\ldots,a_n/a_0)\neq 0\}$. 
Let $F_*(X_1,\ldots,X_n)=F(1,X_1,\ldots,X_n)$ : by what precedes we have $D(F)\cap U_0=u_0^{-1}(D(F_*))$, so $u_0(D(F)\cap U_0)=D(F_*)$, and $u_0^{-1}$ is continuous. As for $u_0$, if $F\in k[X_1,\ldots,X_n]$, $F^*(X_0,\ldots,X_n)=X_0^dF(X_1/X_0,\ldots,X_n/X_0)$ is the homogenisation of $F$ and $(a_1,\ldots,a_n)\in D(F)$, we have $F^*(1,a_1,\ldots,a_n)=F(\ov a)\neq 0$, so $u_0^{-1}(a_1,\ldots,a_n)=[1:a_1:\ldots:a_n]\in D(F^*)\cap U_0$; conversely, if $[\ov a]\in D(F^*)\cap U_0$, we have $[\ov a]=[\ov b]$ for $\ov b=(1,a_1/a_0,\ldots,a_n/a_0)$ and therefore $0\neq F^*(a_0,\ldots,a_n)=a_0^d F(b_1,\ldots,b_n)$, so $F(b_1,\ldots,b_n)\neq 0$ and $u_0([\ov a])=u_0([\ov b])=(b_1,\ldots,b_n)\in D(F)$, so $D(F)=u_0(U_0\cap D(F^*))$, $u_0^{-1}(D(F))=U_0\cap D(F^*)$ and $u_0$ is continuous. Finally, $u_0$ is a homeomorphism.
\end{proof}

\begin{defi}
If $F,G\in k[X_1,\ldots,X_n]$ have respective degrees $d$ and $e$, the \emph{homogenisation of $F/G$} is the rational fraction $(F/G)^*:=X_0^{e-d} F^*(X_0,\ldots,X_n)/G^*(X_0,\ldots,X_n)$, where $F^*$ and $G^*$ are the respective homogenisations of $F$ and $G$.
\end{defi}
\begin{rem}
The homogenisation of $F/G$ is well defined : if $F/G=H/L$ with $u=deg(H)$ and $v=deg(L)$, we have $e-d=v-u$ and $F^*L^*=H^*G^*$, whence $X_0^{e-d}F^*/G^*=X_0^{v-u} H^*/L^*$.
\end{rem}

\noindent Let $U\subs \mbb P^n(k)$ be open : we will say that a function $f:U\to k$ is \emph{regular at $P\in U$} if there exists an open neighbourhood $V\subs U$ of $P$ and $g,h\in k[X_0,\ldots,X_n]$ homogeneous of the same degree, such that for every $Q\in V$, $h(Q)\neq 0$ and $f(Q)=g(Q)/h(Q)$. We note $k[X_0,\ldots,X_n]_h$ the sub-$k$-algebra of $k(X_0,\ldots,X_n)$ consisting of elements of the form $f/g$, with $f$ and $g$ homogeneous of the same degree and $\O(U)$ the set of regular functions on $U$ (i.e. of functions $f:U\to k$ regular ar each $P\in U$), a sheaf on $\mbb P^n(k)$. Finally, we let $\O_i$ be the sheaf of regular functions on $U_i$, the restriction of $\O$ to $U_i$. A general version of Lemma 6.8 and Proposition 6.9 of \cite{MAG}, is the following

\begin{thm}\label{PROJSP}
The projective space $\mbb P^n(k)$, equiped with its sheaf $\O$ of regular functions, is an algebraic equivariety over $k$.
\end{thm}
\begin{proof}
As $\mbb P^n(k)=\bigcup_{i=0}^n U_i$ and each $U_i$ is open, it suffices to show that $(\mbb P^n(k),\O)$ is an equivariety over $k$, and by Proposition \ref{MAG.4.27} that :\\
i) $(U_i,\O_i)$ is an affine algebraic variety over $k$ for each $i$ and $u_i:U_i\cong k^n$ is a regular isomorphism, for each $i$\\
ii) $U_i\cap U_j$ is an open affine subvariety for all $i,j$\\
iii) $\mbb P^n(k)$ is separated.\\
i) It suffices to work with $i=0$; if $U\subs k^n$ is open and $f\in \O(U)=\O_{k^n}(U)$ is regular, we consider the map $f^*=f\circ u_0:[\ov a]\in u_0^{-1} U\mapsto f(\ov b)$, with $[\ov a]=[1:\ov b]$, i.e. $\ov b=b_1,\ldots,b_n$, $b_i=a_i/a_0$ for $i=1,\ldots,n$. 
By regularity of $f$, for every $\ov b\in U$ there exist $G,H\in k[X_1,\ldots,X_n]$, of degrees $d$ and $e$ say, and an open neighbourhood $V\subs U$ of $\ov b$ such that for all $\ov c\in V$, $H(\ov c)\neq 0$ and $f(\ov c)=G(\ov c)/H(\ov c)$. If $[\ov a]=[1:\ov b]\in u_0^{-1}(V)$, we thus have $f^*([\ov a])=f(\ov b)=G(\ov b)/H(\ov b)=1^{e-d} G^*(1,\ov b)/H^*(1,\ov b)=(G/H)^*([1:\ov b])=(G/H)^*([\ov a])$, so the restriction of $f^*$ to $u_0^{-1}(V)\to k$ has a representation as a quotient of homogeneous elements of $k[X_0,\ldots,X_n]$ of the same degree, and therefore $f^*$ is regular on $u_0^{-1}(V)$, i.e. lies in $\O_0(u_0^{-1}(V))$ and as this is true for every $[\ov a]\in u_0^{-1}(U)$, $f^*$ is regular over $u_0^{-1}(U)$ and we have defined a map $\Phi_U:f\in\O(U)\mapsto f^*=f\circ u_0\in\O_0(u_0^{-1}(U))=(u_0)_*\O_0(U)$, which by construction is obviously a morphism of $k$-algebras. With the same notations, if $f^*=\Phi_U(f\in \O(U))\equiv 0$ and $\ov b\in U$, we have $[1:\ov b]\in u_0^{-1}(U)$, so $f(\ov b)=f^*(1,\ov b)=0$, $f\equiv 0$ and therefore $\Phi_U$ is injective. Now suppose $g\in \O_0(u_0^{-1}(U))$ is regular, and consider the map $g_*=g\circ u_0^{-1}:U\to k$, $\ov b\in U\mapsto g([1:\ov b])$; for every $\ov b\in U$, at the neighbourhood $V\subs u_0^{-1}(U)$ of $[1:\ov b]$ by definition there are $H,L\in k[X_0,\ldots,X_n]$, homogeneous of the same degree, such that $g|_V\equiv H/L$, whence $g_*|_{u_0(V)}\equiv H_*/L_*=H(1,\ov x)/L(1,\ov x)$, so $g_*$ is regular, and as obviously we have $\Phi_U(g_*)=g$, $\Phi_U$ is surjective and therefore an isomorphism. As $\Phi:\O_{k^n}\to (u_0)_*\O_0$ is obviously a natural transformation of sheaves, $(u_0,\Phi)$ is an isomorphism of (locally) ringed spaces in $k$-algebras, $(U_0,\O_0)$ is an affine algebraic equivariety, and we conclude that $(\mbb P^n(k),\O)$ is a compact equivariety over $k$.\\
ii) For all $i\neq j$, we have $U_i\cap U_j=\{[\ov a] \in \mbb P^n(k) : a_i\neq 0\ \&\ a_j\neq 0\}$. Under the homeomorphism $u_i:U_i\cong k^n$, we have $u_i(U_i\cap U_j)=\{\ov b=(b_1,\ldots,b_n)\in k^n : b_{j+1}\neq 0\}$ if $j<i$, or $u_i(U_i\cap U_j)=\{\ov b=(b_1,\ldots,b_n)\in k^n : b_j\neq 0\}$ if $j> i$; in both cases, $u_i(U_i\cap U_j)$ is an affine algebraic equivariety by \cite{EQAG}, Lemma 3.15, so $U_i\cap U_j$ is an open affine subvariety of $V$.\\
iii) We want to apply Proposition \ref{MAG.4.27} for the cover $\mbb P^n(k)=\bigcup_i U_i$ and we focus on $i=0$ and $i=1$. The elements $X_1/X_0,\ldots,X_n/X_0$ of $k(X_0,\ldots,X_n)$ are algebraically independent : if $F\in k[X_1,\ldots,X_n]$ and $F[X_1/X_0,\ldots,X_n/X_0]=0$, we have $F^*[X_0,X_1,\ldots,X_n]=X_0^d F[X_1/X_0,\ldots,X_n/X_0]=0$, and therefore $F=F^*[1,X_1,\ldots,X_n]=0$. Thus, the map $k[X_1,\ldots,X_n]\to k[\frac{X_1}{X_0},\ldots,\frac{X_n}{X_0}]\subs k(X_0,\ldots,X_n)$, $X_i\mapsto X_i/X_0$, is an isomorphism of $k$-algebras, and so is $k[X_1,\ldots,X_n]_M\to k[\frac{X_1}{X_0},\ldots,\frac{X_n}{X_0}]_M$, $X_i/1\mapsto (X_i/X_0)/1$, $i=1,\ldots,n$ : by (i) we get a $k$-isomorphism $$\phi:k\left[\frac{X_1}{X_0},\ldots,\frac{X_n}{X_0}\right]_M\cong \O_0(U_0),$$ 
and it is easy to see that $\phi((F/G)[X_1/X_0,\ldots,X_n/X_0]):[\ov a]\in U_0\mapsto (F/G)(a_1/a_0,\ldots,a_n/a_0)$; likewise, we have an isomorphism $$\psi:k\left[\frac{X_0}{X_1},\frac{X_2}{X_1},\ldots,\frac{X_n}{X_1}\right]_M\cong \O_1(U_1)$$ with $\psi((F/G)[X_0/X_1,\ldots,X_n/X_1]) :[\ov a]\in U_1\mapsto (F/G)(a_0/a_1,a_2/a_1,\ldots,a_n/a_1)$. Furthermore, for $U_{01}=U_0\cap U_1=\{[\ov a]\in \mbb P^n(k) : a_0\neq0\ \&\ a_1\neq 0\}$, as $u_0(U_{01})=V_1:=\{\ov b\in k^n : b_1\neq 0\}$ and $\O_{k^n}(V_1)\cong k[X_1,\ldots,X_n]_{\<X_1\>}\cong (k[X_1,\ldots,X_n]_{X_1})_M$ by \cite{EQAG}, Theorem 3.17, we have $$\ms O_0(U_{01})\cong k\left[\frac{X_1}{X_0},\ldots,\frac{X_n}{X_0}\right]_{\<X_1/X_0\>}\cong \left( k\left[\frac{X_1}{X_0},\ldots,\frac{X_n}{X_0}\right]_{X_1/X_0}\right)_M.$$
In general, we also have $k[X_1,\ldots,X_n]_{X_1}\cong k[X_1,\ldots,X_n,1/X_1]\subs k(X_1,\ldots,X_n)$, so $\newline k[\frac{X_1}{X_0},\ldots,\frac{X_n}{X_0}]_{\frac{X_1}{X_0}}\cong k[\frac{X_1}{X_0},\ldots,\frac{X_n}{X_0},\frac{X_0}{X_1}]\subs k(X_0,\ldots,X_n)$, and finally the restriction $\O_0(U_0)\to \O_0(U_{01})$ is isomorphic to the inclusion $$k\left[\frac{X_1}{X_0},\ldots,\frac{X_n}{X_0}\right]_M\subs k\left[\frac{X_1}{X_0},\ldots,\frac{X_n}{X_0},\frac{X_0}{X_1}\right]_M.$$ Likewise, the restriction $\O_1(U_1)\to \O_1(U_{01})$ is isomorphic to the inclusion $$k\left[\frac{X_0}{X_1},\frac{X_2}{X_1},\ldots,\frac{X_n}{X_1}\right]_M\subs k\left[\frac{X_0}{X_1},\frac{X_2}{X_1},\ldots,\frac{X_n}{X_1},\frac{X_1}{X_0}\right]_M$$ and as $\O_0(U_{01})=\O(U_{01})=\O_1(U_{01})$ and $k[\frac{X_1}{X_0},\ldots,\frac{X_n}{X_0},\frac{X_0}{X_1}]=k[\frac{X_0}{X_1},\frac{X_2}{X_1},\ldots,\frac{X_n}{X_1},\frac{X_1}{X_0}]=:A$, these two embeddings respectively correspond to the restrictions of $\O_0(U_0)$ and $\O_1(U_1)$ to $\O(U_{01})$. Now consider the restricted embeddings $\phi_0:k[\frac{X_1}{X_0},\ldots,\frac{X_n}{X_0}]\into A$ and $\psi_0:k[\frac{X_0}{X_1},\frac{X_2}{X_1},\ldots,\frac{X_n}{X_1}]\into A$ : the subset $im(\phi_0)\cup im(\psi_0)$ generates $A$ as a $k$-algebra, so the $k$-algebra morphism $\phi_0\o\psi_0:k[\frac{X_1}{X_0},\ldots,\frac{X_n}{X_0}]\o_k k[\frac{X_0}{X_1},\ldots,\frac{X_n}{X_1}]\to A$, $(f\o g)\mapsto \phi_0(f)\psi_0(g)$ is surjective; it follows that the induced morphism $(\phi_0)_M * (\psi_0)_M:k[\frac{X_1}{X_0},\ldots,\frac{X_n}{X_0}]_M * k[\frac{X_0}{X_1},\ldots,\frac{X_n}{X_1}]_M\cong (k[\frac{X_1}{X_0},\ldots,\frac{X_n}{X_0}]\o_k k[\frac{X_0}{X_1},\ldots,\frac{X_n}{X_1}])_M\to A_M$ of $*$-algebras over $k$ itself is surjective, i.e. $k\{U_0\} * k\{U_1\}\to k\{U_0\cap U_1\}$ is surjective, with the notations of Proposition \ref{MAG.4.27}. As this is true for the same reasons if we replace $0$ and $1$ by any $i,j=0,1,\ldots,n$ (for trivial reasons if $i=j$), by \ref{MAG.4.27}(iii) we conclude that $\mbb P^n(k)$ is separated, i.e. is an algebraic variety over $k$.
\end{proof}
\begin{rem}
i) We could not use, in the proof of (i), the notion of a regular map of concrete equivarieties, because we have to prove that $(U_i,\O_i)$ is indeed an equivariety. \emph{A posteriori}, the homeomorphisms $u_i$ are regular isomorphisms.\\
ii) This generalisation to any commutative field, of the description of the structure of an algebraic variety on the projective spaces over an algebraically closed field, is possible thanks to the theory of $*$-algebras and canonical localisations of \cite{EQAG} and their essential relation to the algebras of sections of regular functions over affine algebraic subvarieties : this confirms the fertility and the strength of the equiresidual point of view. 
\end{rem}

\subsection*{Locally closed subvarieties of separated and algebraic equivarieties}
Let $V=\bigcup_i U_i$ be a separated (resp. algebraic) equivariety, where the $U_i$ are open affine subvarieties, and $Z$ a closed subvariety of $V$ : we have $Z=\bigcup_i Z_i$ if $Z_i=Z\cap U_i$ for each $i$, and as $Z_i$ is open in $Z$, the restriction $\O_{Z_i}$ of $\O_Z$ to $Z_i$ is simply given by $\O_{Z_i}(U\subs Z_i):=\O_Z(U)$ for each $U\subs Z_i$ open (for the induced topology), and we want to show that this defines the structure of a separated (resp. algebraic) equivariety on $Z$. First, as for each $i$, $Z_i$ is closed in $U_i$, by Lemma \ref{CLSVAR} $Z_i$ is naturally an affine equivariety, so $Z$ is an equivariety, and we want to show that it is separated (resp. algebraic).  

\begin{lem}\label{INDCLS}
If $V$ and $W$ are two equivarieties and $X\subs V$, $Y\subs W$ closed subvarieties, then the topology induced on $X\xx Y$ by the Zariski topology on $V\xx W$ is the Zariski topology on $X\xx Y$.
\end{lem}
\begin{proof}
First, we assume that $V$ and $W$ are affine : by Lemma \ref{CLSVAR}, $X$ and $Y$ are affine as well. Now if $F\subs V\xx W$ is a closed subset, there exists $S\subs J(V)*J(W)$ such that $F=\{(P,Q)\in V\xx W : \forall \sum_i f_i*g_i\in S,\sum_i f_i(P)g_i(Q)=0\}$, so that $F\cap (X\xx Y)=\{(P,Q)\in X\xx Y : \forall \sum_i f_i*g_i\in S,\sum_i f_i|_{X}(P)g_i|_Y(Q)=0\}$, which is clearly closed in $X\xx Y$ for the Zariski topology. Conversely, if $F\subs X\xx Y$ is Zariski closed, it has the form $F=\{(P,Q)\in X\xx Y : \forall \sum_i f_i*g_i\in J(X)*J(Y),\sum_i f_i(P)g_i(Q)=0\}$, and as all $f\in J(X),g\in J(Y)$ are restrictions of elements of $J(V),J(W)$, $F$ is in turn clearly closed for the induced topology.
In the general case, write $V=\bigcup_i V_i$ and $W=\bigcup_j W_j$ as open covers by affine subvarieties. Let $F\subs V\xx W$ be closed for the induced topology : we want to show that $F\cap (X\xx Y)$ is closed for the Zariski topology; as $X\xx Y=\bigcup_{i,j} X_i\xx Y_j$ with $X_i=X\cap V_i$ (open in $X$) and $Y_j=Y\cap W_j$ (open in $Y$) for all $i,j$, it suffices to show that $F\cap (X_i\xx Y_j)$ is Zariski closed for all $(i,j)$, because the subsets $X_i\xx Y_j$ are Zariski open by Lemma \ref{ZARIND} and cover $X\xx Y$. Now for any pair $(i,j)$, we have $F\cap (X_i\xx Y_j)=(X_i\xx Y_j)\cap (V_i\xx W_j)\cap F$, and $F\cap (V_i\xx W_j)$ is certainly closed for the Zariski topology on $V_i\xx W_j$, and by the affine case, we conclude that $F\cap (X_i\xx Y_j)$ is Zariski closed, so that $F\cap (X\xx Y)$ is Zariski closed. Conversely, assume $F\subs X\xx Y$ is Zariski closed; as $X$ and $Y$ are closed, $X\xx Y$ is closed in $V\xx W$ because $(X\xx Y)\cap (V_i\xx W_j)=X_i\xx Y_j$ is Zariski closed in $V_i\xx V_j$ for all $(i,j)$. Now for each pair $(i,j)$, let $F_{ij}=F\cap (X_i\xx X_j)$ : $F_{ij}$ is Zariski closed in $X_i\xx Y_j$, so by the affine case again, it is closed in $V_i\xx W_j$, whence $F$ is closed in $V\xx W$, so $F$ is closed in $X\xx Y$ for the induced topology, and the proof is complete.
\end{proof}

\noindent By Proposition \ref{MAG.4.24} and Lemma \ref{INDCLS}, if $V$ is separated the subset $\Delta_Z=\Delta_V\cap Z^2$ is closed in $Z^2$ for the Zariski topology, so $Z$ is a separated equivariety. We have proved the first part of the

\begin{prop}\label{CLOSALGVAR}
If $V$ is a separated (resp. algebraic) equivariety and $Z$ is a closed subset of $V$, then $(Z,\O_V|_Z)$ is a separated (resp. algebraic) equivariety as well.
\end{prop}
\begin{proof}
By what precedes, the separated case is established, so we assume that $V$ is algebraic, i.e. separated and compact. We know that $Z$ is separated, and keeping the same notations as before, we may assume by compactness that the affine open cover $V=\bigcup_i V_i$ is finite, and thus the affine open cover $Z=\bigcup_i Z_i$ is finite as well, so $Z$ is compact, therefore it is an algebraic equivariety.
\end{proof}

\begin{thm}\label{LOCSEP}
If $(V,\O_V)$ is a separated (resp. algebraic) equivariety and $S\subs V$ is a locally closed subset, then $(S,\O_V|_S)$ is a separated (resp. algebraic) equivariety as well.
\end{thm}
\begin{proof}
As before we assume that $V$ is concrete, and we first assume that $S$ is open in $V$. Write $V=\bigcup_i U_i$ with $U_i$ an affine open subvariety for each $i$ and let $S_i=U_i\cap S$ for each $i$ : each $S_i$ is open in $S$ and a quasi-affine algebraic equivariety; by Proposition \ref{QAFFALGVAR}, $S_i$ is an algebraic equivariety, and therefore $S$ is a union of affine algebraic equivarieties, each open in $S$, so $S$ is an equivariety over $k$. As for separation, as in \ref{QAFFALGVAR} if $\phi_1,\phi_2:U\to S$ are regular maps with $U$ affine, and $i:S\into V$ is the inclusion, the set $\{u\in U : \phi_1(u)=\phi_2(u)\}=\{u\in U : i\phi_1(u)=i\phi_2(u)\}$ is closed because $V$ is separated, so $S$ itself is separated. Taking a finite affine open cover of $V$ if $V$ is algebraic, we see that $S$ is compact in this case, so itself an algebraic equivariety. 
In general, i.e. if $S$ is locally closed and $V$ is separated (resp. algebraic), there exist a closed $F\subs V$ and an open $O\subs V$ such that $S=F\cap O$; by Proposition \ref{CLOSALGVAR}, $(F,\O|_F)$ is a separated (resp. algebraic) equivariety, and by the open case, 
$(S,(\O|_F)|_S)$ is also a separated (resp. algebraic) equivariety, whence the result, because $(\O|_F)|_S=\O_V|_S$ by Lemma \ref{DBRES}. 
\end{proof}

\noindent As in \og absolute" algebraic geometry (i.e. over algebraically closed fields), we define a \emph{quasi-projective equivariety} as a locally ringed space in $k$-algebras which is isomorphic to a locally closed subvariety of a projective space over $k$ (we do \emph{not} restrict ourselves to irreducible equivarieties). All what precedes gives us the following

\begin{cor}
Every quasi-projective variety over $k$ is an algebraic equivariety. 
\end{cor}

\section{Conclusions}
The basic theory developped in \cite{EQAG} proves to be robust at the geometric level : we have been able, quite straightforwardly, to generalise a theory of algebraic varieties over an algebraically closed field, as expounded in \cite{MAG}, Chapter 4, to a theory of algebraic equivarieties over an arbitrary (commutative) field. Several things are in fact naturally valid in the general context of equivarieties, and the basic theory of projective spaces and quasi-projective (equi)varieties fits neatly into this framework. This justifies to carry on the program evoked at the end of \cite{EQAG}, and for a start we will focus on the local study of algebraic equivarieties, the definition of an equiresidual version of étale regular morphisms and its connexion with the usual theory of étale ring morphisms and henselisation, and the possible extension of a notion of integral dependence in connexion with a possible notion of a normal algebraic equivariety.\\

The original theory stemed from considerations about fundamental connexions between positive model theory and algebraic geometry. The theory presented in \cite{EQAG} and here will allow us to fully develop in \cite{PAG1} the first stage of \emph{positive algebraic geometry}, an interpretation of equiresidual algebraic geometry in the context of positive logic, where we will generalise some \og definable" aspects of algebraic geometry over algebraically closed fields in certain first order theories of fields, using for instance the logical setting in order to give systematic axiomatisations of classes of algebraic objets - like the counterparts of $*$-algebras, coordinate rings or function fields - much in the spirit of universal algebra and coherent logic. On this quasi-axiomatisable version of (equi)algebraic geometry we we will hopefully build, thanks to the tools of coherent logic, a general theory of \og fields with an open definable structure", i.e. in which a certain analogue of Tarski-Chevalley-Macintyre's theorem, valid in algebraically closed fields, real closed fields and $p$-adically closed fields (see \cite{BEL2} for instance), holds in connexion with some interesting infinitesimal properties.

\vspace{0.5cm}
\textit{\footnotesize E-mail adress} : \texttt{\footnotesize jeb.math@gmail.com}
\end{document}